\documentclass[reqno,twoside]{amsart}
\usepackage{amsmath}
\usepackage{amsfonts}
\usepackage{amssymb,amsthm}
\usepackage{enumerate}
\usepackage{placeins}
\usepackage{mathtools}
\usepackage{float}
\usepackage[hypcap=false]{caption}
\makeatletter
\@namedef{subjclassname@2020}{\textup{2020} Mathematics Subject Classification}
\makeatother

\usepackage{a4wide,amsmath,amssymb,latexsym,amsthm}

\usepackage[colorlinks=true,urlcolor=blue,
citecolor=red,linkcolor=blue,linktocpage,pdfpagelabels,
bookmarksnumbered,bookmarksopen]{hyperref}
\usepackage[english]{babel}

\numberwithin{equation}{section}
\usepackage{mathrsfs,mathtools,epic,bm}
\usepackage{hyperref}
\hypersetup{colorlinks=true,linkcolor=blue,filecolor=mangeta,urlcolor= cyan}

\newtheorem{theorem}{Theorem}[section]
\newtheorem{proposition}[theorem]{Proposition}
\newtheorem{lemma}[theorem]{Lemma}

\theoremstyle{definition}
\newtheorem{definition}[theorem]{Definition}
\newtheorem{example}[theorem]{Example}
\newtheorem{remark}[theorem]{Remark}

\newtheorem{result}[theorem]{Result}

\newcommand\abs[1]{\left\lvert#1\right\rvert}

\numberwithin{equation}{section}

\def \dis {\displaystyle}
\def \R {\mathbb{R}}

\def \dx{\mathrm{d}x}

\def \ds{\mathrm{d}s}

\def \dq {\mathrm{d}x \, \mathrm{d}t}

\def \H {\mathbb{H}}

\def \V {\mathbb{V}}
\def \L {\mathbb{L}}

\def \H {\mathbb{H}}

\def \U {\mathcal{U}_{ad}}
\def \W {\mathbb{W}}

\keywords{Optimal control problems, Fokker-Planck equation, Metric graphs, shifted Legendre wavelets, Collocation method}

\begin{document}
	\title[Fokker--Planck Dynamics on Star Graphs]
{Fokker--Planck Dynamics on Star Graphs with Variable Drift:
Well-Posedness, Adjoint Analysis, and Numerical Approximation}

\author{Ritu Kumari}
\address{Department of Mathematics, Indian Institute of Technology Delhi,
New Delhi 110016, India}
\email{ritsrishu98@gmail.com}

\author{Cyrille Kenne}
\address{Department of Mathematics, University of British Columbia,
Vancouver, BC V6T 1Z2, Canada}
\email{kenne@math.ubc.ca}

\author{Landry Djomegne}
\address{Laboratoire LAMIA, D\'epartement de Math\'ematiques et Informatique,
Universit\'e des Antilles, Campus Fouillole,
97159 Pointe-\`a-Pitre, Guadeloupe}
\email{landry.djomegne@gmail.com}

\author{Mani Mehra}
\address{Department of Mathematics, Indian Institute of Technology Delhi,
New Delhi 110016, India}
\email{mmehra@maths.iitd.ac.in}

	\date{\today}

	
\begin{abstract}
	Stochastic transport processes on networked domains (modelled on metric graphs)  arise in a variety of applications where diffusion and drift mechanisms interact with an underlying graph structure. The Fokker--Planck equation provides a natural framework for describing the evolution of probability densities associated with such dynamics. While Fokker--Planck equations on metric graphs have been studied from an analytical viewpoint, their optimal control remains largely unexplored, particularly in settings where the control acts through the drift term.
In this paper, we investigate an optimal control problem governed by the Fokker--Planck equation on a star graph, with a bilinear control appearing in the drift. We establish the well-posedness of the state equation and prove the existence of at least one optimal control. The associated adjoint system is derived, and first-order necessary optimality conditions are formulated. A wavelet-based numerical scheme is proposed to approximate the optimal solution, and its performance is illustrated through representative numerical experiments. These results contribute to the analytical and computational understanding of controlled stochastic dynamics on network-like domains.
\end{abstract}
	\maketitle	
\section{Introduction}

The Fokker-Planck equation is a partial differential equation that describes the time evolution of the probability density function of a particle's velocity, particularly when the particle is subjected to both drag and stochastic (random) forces, as observed in Brownian motion. This equation can be extended to describe the evolution of other observables as well \cite{kadanoff2000statistical}. Owing to its generality, the Fokker-Planck equation finds broad applications in diverse fields such as fluid dynamics \cite{moroni2006}, physics \cite{escande2007, escande2008,  malkov2017}, socio-physics \cite{cordier2005, jagielski2013, yakovenko2007} and even in neural networks \cite{huang2005, kamitani2004, suykens1998}.

 Many physical systems are traditionally modeled using differential equations in Euclidean domains. However, several real-world processes, such as gas transport in pipeline networks \cite{steinbach2007pde}, water wave propagation in open channel systems described by Burgers' equation \cite{yoshioka2014burgers}, and fluid motion through canal networks governed by the Saint-Venant equations \cite{xu2010differential}, are more naturally represented by partial differential equations (PDE) defined on metric graphs. 
  In particular, the Fokker-Planck equation on graphs has drawn growing attention for its effectiveness in modeling stochastic processes across network-like structures (see e.g., \cite{chow2018entropy, chow2012fokker} and the references therein). For example, in neuroscience, it provides a powerful framework to describe how neurotransmitters diffuse or how membrane potentials distribute along dendritic trees, where each dendrite is represented as an edge in the neural network. In the context of chemical and biological systems, it models molecular transport through branched capillary or vascular networks. Furthermore, in financial mathematics, metric graphs may represent networks of interlinked markets or financial instruments, allowing the Fokker-Planck framework to describe the evolution of risk distributions over such structures. The explicit solution of the Fokker-Planck equation on metric graphs has been presented in \cite{MATRASULOV2022128279}, emphasizing its growing relevance.

  While the Fokker--Planck equation on metric graphs provides a natural framework for describing the evolution of probability densities over networked domains, many applications require more than a passive description of the underlying stochastic dynamics. In several practical situations, it is desirable to actively influence the evolution of the probability distribution by acting on external parameters of the system. A particularly relevant mechanism consists in modifying the drift term, which allows one to steer the stochastic process along the network in a controlled manner. This leads naturally to the formulation of optimal control problems in which the Fokker--Planck equation serves as the state constraint and the control enters in a bilinear fashion through the drift. From both modeling and analytical viewpoints, such problems provide a systematic framework to balance control costs against prescribed objectives, while accounting for the interaction between stochastic dynamics and the geometry of the underlying graph. These considerations motivate the study of optimal control problems governed by Fokker--Planck equations on metric graphs.

\subsection{Literature review}

Optimal control problems (OCPs) involve optimizing a prescribed cost functional on a set of admissible control strategies, subject to dynamic constraints typically expressed as differential equations \cite{kirk2004optimal}. A classical example appears in aerospace engineering, where OCPs are used to compute an optimal thrust trajectory that minimizes fuel consumption while achieving the insertion in a desired orbit \cite{morante2021survey,olympio2011optimal, trelat2012optimal}. The optimal control of partial differential equations (PDEs) on graphs has emerged as an active area of research, as it provides insight not only into the evolution of probabilistic systems over networked structures, but also guides these systems toward desired states. In \cite{mehandiratta2021optimal}, the authors analyzed optimal control problems governed by time-fractional diffusion equations on metric graphs, established the well-posedness of the model, and implemented a finite difference method for numerical approximation. Parabolic fractional initial-boundary value problems of Sturm-Liouville type on both intervals and general star graphs were addressed in \cite{leugering2022optimalcontrolproblemsparabolic}, where results on the existence, uniqueness and regularity of weak and very weak solutions were obtained. Similarly, the existence and regularity of solutions for the fractional wave equation of Sturm-Liouville type on star graphs with mixed Dirichlet and Neumann boundary conditions were derived in \cite{Moutamal02122023}. Additional studies on optimal control problems on star graphs can be found in \cite{dorville2011optimal,graef2014existence, mehandiratta2019existence, mehandiratta2021existence,  stoll2019}. In \cite{mohammadi2015hermite}, a Hermite spectral discretization scheme was proposed to approximate the solution of a Fokker-Planck optimal control problem in an unbounded interval. We mention that the optimal control of the Fokker-Planck equation has also been motivated by mean field game theory
in the recent years. For more details, we refer the reader to \cite{achdou2021, achdou2020, camilli2016, lasry2007}.

In this paper, an optimal bilinear control problem is considered on a metric graph, where the control appears in the drift term. The well-posedness of the system is established, and the existence of an optimal control is demonstrated. In addition, the adjoint system and the associated first-order optimality conditions are derived, and the theoretical findings are later supported through numerical simulations. For optimal control problems related to the Fokker-Planck equation in the Euclidean domain, we refer the reader to \cite{ annunziato2013, aronna2021,fleig2017, kenne2026,roy2016}.

 To numerically validate the theoretical findings, a wavelet-based method is utilized, which is well-suited for solving partial differential equations due to their inherent advantages such as multiscale decomposition, localization properties, and orthogonality. A related approach was presented in \cite{shukla2020fast}, where Burgers' equation on an open channel network was addressed using a fast adaptive spectral graph method for efficient numerical simulation.
 In \cite{doi:10.1177/10775463231169317}, a generalized fractional-order Legendre wavelet method was employed to solve a two-dimensional distributed-order fractional optimal control problem, demonstrating the efficiency of the approach. Similarly, the time-fractional diffusion equation was first solved using a finite difference scheme in \cite{MEHANDIRATTA2020152}, and later revisited in \cite{faheem2023collocation} using a Haar wavelet collocation method, which outperformed the former approach. Motivated by the advantages of wavelet-based techniques, a wavelet collocation method for solving the optimal control problem is proposed in this paper. To the best of our knowledge, this is the first time, when a wavelet collocation method has been applied to optimal control problems on graphs.

\subsection{Problem Formulation}
To develop the mathematical framework on networked domains, let $\mathcal{G} = (V, E)$ denote a finite graph, where $V = \{v_i\}_{i=0}^N$ is the set of vertices (nodes) and $E$ is the set of edges connecting them. The graph $\mathcal{G}$ is referred to as a metric graph if each edge is equipped with a corresponding metric structure, thereby allowing differential operators to be defined along the edges. The main purpose of this paper is to study the optimal control of the Fokker--Planck equations on a metric star graph $\mathcal{G}$ with $N$ edges (see Figure \ref{stargraph}). The controls appear in the drift term of the Fokker--Planck equation. To be more precise, we consider the following control problem.
             

\begin{equation}
 \label{opt}
 \begin{array}{lll}
  \dis   \inf_{u\in \U}J(\rho,u)&:=& \dis \frac{1}{2}\sum_{i=1}^{N}\int_{Q_i} \abs{\rho_i-\rho^d_i}^2\,\dq+\frac{1}{2}\sum_{i=1}^{N}\int_{0}^{l_i} \abs{\rho_i(T)-\rho^T_i}^2\,\dx\\
&&
\dis +\frac{1}{2}\sum_{i=1}^{N}\alpha_i\int_{Q_i}
|u_i|^{2} \,\dq, 
 \end{array}
\end{equation}
\normalsize
where for $i=1,\ldots, N$, $\rho_i$ is the state associated with the control $u_i$, solution of
\begin{equation}\label{model}
\left\{
\begin{array}{lllllllllllllllll}
\displaystyle	\frac{\partial \rho_i}{\partial t}-D_i\frac{\partial^2 \rho_i}{\partial x^2}-\frac{\partial }{\partial x}\left(u_i\rho_i\right)&=&0&\text{in}& Q_i,\\[6pt]
\displaystyle	\rho_i(x,0)&=&\rho^0_i(x),& &x\in(0,l_i),\\[3pt]
\displaystyle	\sum_{i=1}^N\left[D_i\frac{\partial \rho_i}{\partial x}(0,t)+u_i(0,t)\rho_i(0,t)\right]&=&0,& & \\[3pt]
\displaystyle \rho_i(0,t)&=&\rho_j(0,t), & i\neq j,& t\in (0,T),\\ [3pt]
\displaystyle	\rho_i(l_i,t)&=&0,& & 
\end{array}
\right.
\end{equation}
with $Q_i:=(0,l_i)\times (0,T)$.  Here, $\rho^d=(\rho_i^d)_i \in L^2(0,T;\prod_{i= 1}^N L^2(0,l_i))$, $\rho^T=(\rho_i^T)_i \in \prod_{i= 1}^N L^2(0,l_i)$, $\alpha_i>0$, $D_i>0$,  $i=1,...,N$ and the set of admissible controls is defined as
\begin{equation}\label{defuad}
\U:=\left\{u=(u_i)_i\in \prod_{i= 1}^N L^\infty(Q_i): u^{\min}\leq u(x,t)\leq u^{\max}\;\; \text{a.e.}\; \text{in}\; \prod_{i= 1}^N Q_i\right\},
\end{equation}
where $u^{min}=(u_i^{\min})_i$, $u^{\max}=(u_i^{\max})_i\in \prod_{i= 1}^N L^\infty(Q_i)$,  are arbitrary but fixed. The initial condition $ \rho^0=(\rho^0_i)_i\in \prod_{i= 1}^N L^2(0,l_i)$.
This optimal control problem \eqref{opt} governed by the Fokker-Planck equation \eqref{model}-\eqref{defuad} will be called an FPOCP throughout the paper.
\begin{figure}[!h] 
            \centering
            \includegraphics[scale=0.24]{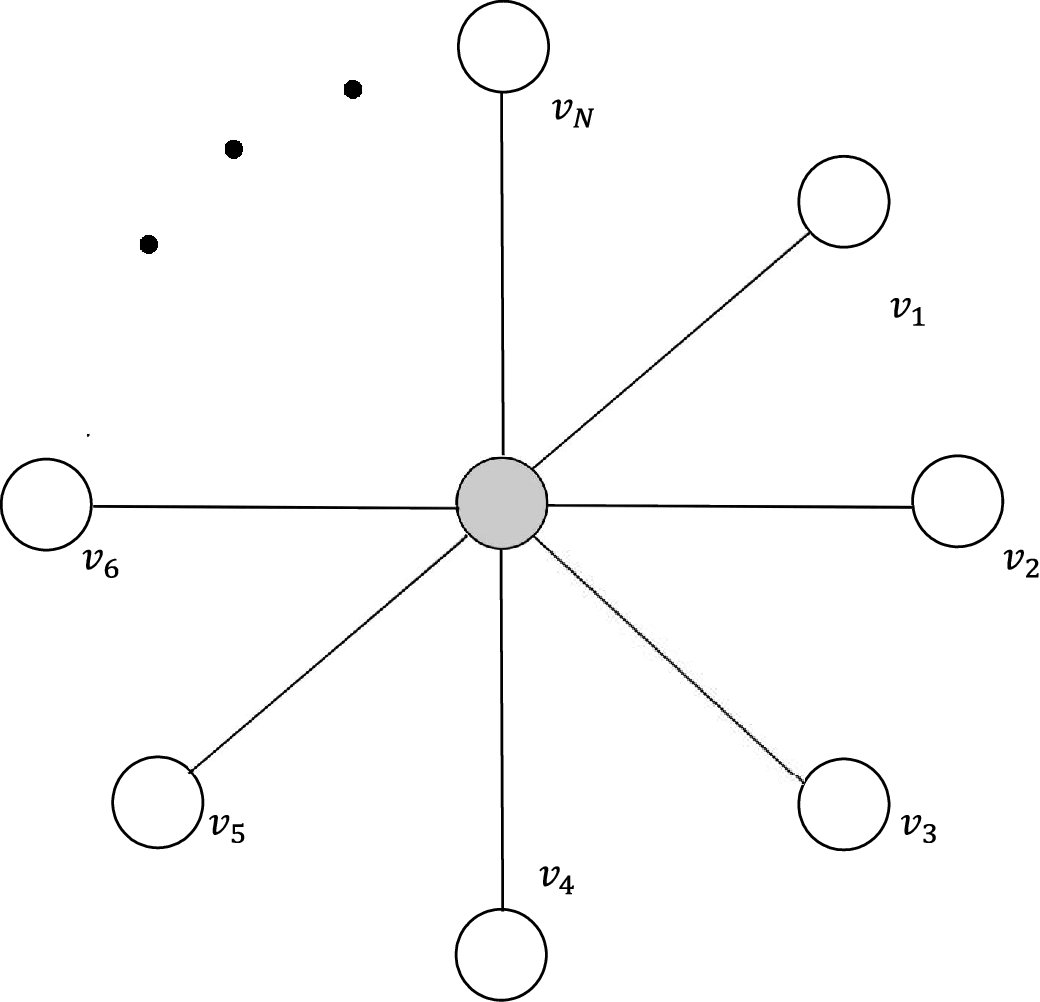} 
            \caption{A star graph having $N$ edges.}
            \label{stargraph}
    \end{figure}

 \subsection{Contribution}
 There are two primary approaches to solve an optimal control problem. The first is the direct method, in which the problem is directly transformed into an optimization problem and subsequently solved using a nonlinear programming (NLP) solver. The second is the indirect method, where the necessary optimality conditions are derived using the calculus of variations. In this work, the indirect method is employed, where the variations of the augmented functional are computed with respect to the state, costate, and control variables. The main contributions of this work are summarized as follows:
\begin{itemize}
    \item The existence and uniqueness of the weak solution to the Fokker-Planck equation is established, and the existence of an optimal solution to the proposed optimal control problem (FPOCP) on a star graph is proved.
    
    \item The corresponding adjoint system and first-order necessary optimality conditions are derived, thereby reformulating the FPOCP into a system of nonlinear algebraic equations.
    
    \item A numerical method based on shifted Legendre wavelets \cite{kumari2025muntz} is developed to approximate the optimal solution of the FPOCP. Numerical experiments are provided to validate the theoretical results and to demonstrate the efficiency of the proposed method.
\end{itemize}
The structure of the paper is as follows: Section~\ref{prelim} presents preliminary concepts and introduces the fundamental notations used throughout the paper. In Section~\ref{existence}, the existence and uniqueness of a weak solution to the nonlinear, non-homogeneous Fokker-Planck equation on graphs are established. Section~\ref{exopt} proves the existence of an optimal solution to the FPOCP by introducing a control-to-state operator and deriving the corresponding adjoint equation. Additionally, the first-order necessary optimality condition is stated in this section. In Section~\ref{numerical illus}, the shifted Legendre wavelets are briefly introduced, and based on these wavelets a numerical method  is developed to solve the optimality system associated with the FPOCP. Section~\ref{numerical res} presents two numerical examples to demonstrate the effectiveness of the proposed method. Finally, Section~\ref{conc} concludes the paper.

\section{Preliminary results\label{prelim}}
\subsection*{Notations}
We set $\mathbb{L}^p := \prod_{i= 1}^NL^p(0,l_i)$, $\mathbb{L}^\infty := \prod_{i= 1}^NL^\infty(Q_i)$ \\
and $\H^1:=\prod_{i= 1}^NH^1(0,l_i)$. We endow these spaces by the following norms $\|u\|_{\L^p} := \left(\sum_{i=1}^{N}\|u_i\|^p_{L^p(0,l_i)}\right)^{\frac{1}{p}}$,  $\|u\|_{\L^\infty} := \sum_{i=1}^{N}\|u_i\|_{L^\infty(Q_i)}$, and $\|u\|^2_{\H^1} := \sum_{i=1}^{N}\|u_i\|^2_{H^1(0,l_i)}$, respectively.  Let $(\H^1)^\star$ be the topological dual space of $\H^1$ and we denote by $\left\langle \cdot, \cdot \right\rangle$ the dual product between $\H^1$ and $(\H^1)^\star$. We will also denote by $\left(\cdot, \cdot \right)_{\H^1}$ and $\left(\cdot, \cdot \right)$ the usual scalar products on $\H^1$ and $\L^2$, respectively. We also set $\dis \bar{D}=\max_{i=1,...,N}D_i$ and $\dis \underbar{D}=\min_{i=1,...,N}D_i$.

Let us introduce the space 
$$
\mathbb{V}:=\left\{ \varphi\in \mathbb{H}^1: \varphi_i(0)=\varphi_j(0), \; i\neq j, \varphi_i(l_i)=0,\; i,j=1,...,N  \right\}.
$$
Then $\mathbb{V}$ is a closed subspace of the Hilbert space $\mathbb{H}^1$. Therefore, $\mathbb{V}$ endowed with the norm of $\mathbb{H}^1$ is also a Hilbert space. We denote by $\left\langle \cdot, \cdot \right\rangle_{\mathbb{V}, \mathbb{V}^\star}$ the dual product between $\mathbb{V}$ and $\mathbb{V}^\star$ and by identifying $\L^2$ with its dual $(\L^2)^\star$, we have $\mathbb{V}\hookrightarrow \L^2\hookrightarrow \mathbb{V}^\star$, and $
	(\rho, \varphi)=\left\langle \rho, \varphi \right\rangle_{\mathbb{V}, \mathbb{V}^\star},\;\; \rho \in \L^2, \; \varphi\in \mathbb{V}.
$
If we set
$
\W(0,T)= \left\{\rho \in L^2(0,T;\V); \rho_t\in L^2\left(0,T;\V^\star\right)\right\},
$
then $\W(0,T)$ endowed with the norm
$
\|\rho\|^2_{	\W(0,T)}=\|\rho\|^2_{L^2(0,T;\V)}+\|\rho_t\|^2_{L^2\left(0,T;\V^\star\right)}
$
is a Hilbert space. Moreover, we have that the embedding 
$
\W(0,T)\hookrightarrow  \mathcal{C}([0,T];\L^2),
$ is continuous
and the embedding
$
\W(0,T)\subset \subset L^2(0,T;\L^2)
$
is compact.

\section{Existence results}\label{existence}
In this section, we prove the existence and uniqueness of the following nonlinear non-homogeneous Fokker-Planck equation on a graph. Let $\rho^0=(\rho^0_i)_i\in \L^2$, $u=(u_i)_i\in \L^\infty$ and $f=(f_i)_i\in L^2(0,T;\V^\star)$. We consider the following problem:
\begin{equation}\label{general}
\left\{
\begin{array}{lllllllllllllllll}
\displaystyle	\frac{\partial \rho_i}{\partial t}-D_i\frac{\partial^2 \rho_i}{\partial x^2}-\frac{\partial }{\partial x}\left(u_i\rho_i\right)&=&f_i&\text{in}& Q_i,\\[6pt]
\displaystyle	\rho_i(x,0)&=&\rho^0_i(x),& &x\in(0,l_i),\\[3pt]
\displaystyle	\sum_{i=1}^N\left[D_i\frac{\partial \rho_i}{\partial x}(0,t)+u_i(0,t)\rho_i(0,t)\right]&=&0,& & \\[3pt]
\displaystyle \rho_i(0,t)&=&\rho_j(0,t), & i\neq j,& t\in (0,T),\\ [3pt]
\displaystyle	\rho_i(l_i,t)&=&0,& &
\end{array}
\right.
\end{equation}
with $i,j=1,\ldots,N$. We define the weak solution to the system \eqref{general} as follows.
\begin{definition}\label{weakgen}
	Let $f\in L^2(0,T;\V^\star)$, $u\in \L^\infty$  and $\rho^{0}\in \L^2$. We  say that a function
	$\rho=(\rho_i)_i\in \W(0,T)$ is a weak solution to \eqref{general}, if:
	\begin{equation}\label{weakgene}
	\begin{array}{lll}
	\dis \sum_{i=1}^{N}\frac{d}{dt}\left(\rho_i(\cdot),\varphi_i\right) + \mathcal{F}(\rho(\cdot),\varphi) = \sum_{i=1}^{N}\left\langle f_i(\cdot),\varphi_i\right\rangle \;\text{in }\; \mathcal{D}'(0,T), \,\forall \varphi=(\varphi)_i\in \V,\\
	\dis  \rho(\cdot,0)=\rho^0 \text{ in } \L^2,
	
	\end{array}
	\end{equation}
	where, $\mathcal{F}:=\mathcal{F}[u]: \V\times \V\to \R$ is defined by
	\begin{equation}\label{defFt}
	\dis	\mathcal{F}(\rho,\varphi):= \sum_{i=1}^{N}\int_{0}^{l_i} \left(D_i\frac{\partial \rho_i}{\partial x}+u_i\rho_i\right)\frac{\partial \varphi_i}{\partial x} \, \dx.
	\end{equation}
	
\end{definition}
In the rest of the paper, we denote by $C$ a positive constant that may vary line to line. Throughout the rest of the paper, $i,j=1,\ldots,N$, unless stated otherwise. We have the following first main result of this section.

\begin{theorem}\label{theoremexistence}
		Let $f\in L^2(0,T;\V^\star)$, $u\in \L^\infty$  and $\rho^{0}\in \L^2$.  Then, there exists a unique weak solution $\rho=(\rho_i)_i\in \W(0,T)$ to \eqref{general} in the sense of Definition~\ref{weakgen}.
	In addition, there exists a constant $C>0$ depending on $T$, $\underbar{D}$, $\bar{D}$, and continuously on $\|u\|_{\L^\infty}$ such that
	\begin{align}\label{estimationgen}
	\|\rho\|_{\mathcal{C}([0,T];L^2(\Omega))}+ \|\rho\|_{L^2(0,T;\V)} &\leq C\left(\|\rho^0\|_{\L^2}+\|f\|_{L^2(0,T;\V^\star)}\right),\\
\label{estimationge}
	\|\rho\|_{\W(0,T)}&\leq  C\left(\|\rho^0\|_{\L^2}+\|f\|_{L^2(0,T;\V^\star)}\right).
	\end{align}
\end{theorem}

\begin{proof}We proceed in three steps.
	
	\noindent \textbf{Step 1.} 	For $\rho, \varphi \in \V$, a straightforward computation leads us to 
	\begin{equation}\label{contGt}
	\dis	|\mathcal{F}(\rho,\varphi)|\leq C \|\rho\|_{\V} \|\varphi\|_{\V},
	\end{equation}
	where $\dis C:=\bar{D}+ \|u\|_{\L^\infty}$. Moreover, we observe that for every $\rho ,\varphi \in \V$ and thanks to Young's inequality we have
	\begin{equation}\label{ineq0}
	\dis	\left| \sum_{i=1}^{N}\int_{0}^{l_i} u_i\rho_i\frac{\partial \varphi_i}{\partial x} \, \dx\right|\leq \|u\|_{\L^\infty}\left[\frac{\delta}{2}\|\rho\|^2_{\L^2}+\frac{1}{2\delta}\|\varphi\|^2_{\V}\right],
	\end{equation}
	for some $\delta>0$.
	By choosing $\delta>\frac{1}{2 \underbar{D}}\|u\|_{\L^\infty}$ and using the latter inequality, we obtain after a straightforward calculation that for every $\rho\in \V$,
	\begin{equation}\label{coercGt}
	\dis\mathcal{F}(\rho,\rho)\geq \gamma\|\rho\|^2_{\V}-\lambda \|\rho\|^2_{\L^2},
	\end{equation}
	with $\gamma:=\underbar{D}-\frac{1}{2\delta}\|u\|_{\L^\infty}>0$ and $\lambda=\underbar{D}+\frac{\delta}{2}\|u\|_{\L^\infty}>0.$
	Therefore, applying \cite[Theorem 4.1 and Remark 43]{lions1968} (see also \cite[Theorem 11.7]{chipot2000}) along with the continuity and the weak-coerciveness results given in \eqref{contGt} and \eqref{coercGt},  respectively, we deduce the existence of a unique weak solution $\rho\in \W(0,T)$ to \eqref{weakgene}.
	
	\noindent \textbf{Step 2.}
	We show the estimates \eqref{estimationgen}-\eqref{estimationge}:
We use the classical change of variables $\eta_i=e^{-rt}\rho_i$ where $\rho$ is solution to \eqref{general} and for  some $r>0$. Then, $\eta=(\eta_i)_i\in \W(0,T)$ is solution to
	\begin{equation}\label{generalr}
	\left\{
	\begin{array}{lllllllllllllllll}
	\displaystyle	\frac{\partial \eta_i}{\partial t}-D_i\frac{\partial^2 \eta_i}{\partial x^2}-\frac{\partial }{\partial x}\left(u_i\eta_i\right)+r\eta_i&=&e^{-rt}f_i&\text{in}& Q_i,\\[6pt]
	\displaystyle	\eta_i(x,0)&=&\rho^0_i(x),& &x\in(0,l_i),\\[3pt]
	\displaystyle	\sum_{i=1}^N\left[D_i\frac{\partial \eta_i}{\partial x}(0,t)+u_i(0,t)\eta_i(0,t)\right]&=&0,& & \\[3pt]
	\displaystyle \eta_i(0,t)&=&\eta_j(0,t), & i\neq j,& \\ [3pt]
	\displaystyle	\eta_i(l_i,t)&=&0.& &
	\end{array}
	\right.
	\end{equation}
	Next, we test the weak formulation of \eqref{generalr} with $\varphi=\eta$, integrating over $(0,t)$ for some $t\in (0,T)$ to obtain:
	\begin{equation}\label{eq0}
	\begin{array}{lll}
	\dis \frac{1}{2}\|\eta(t)\|^2_{\L^2} + \sum_{i=1}^{N}\int_{0}^{t}\int_{0}^{l_i}D_i\left|\frac{\partial \eta_i }{\partial x}(s)\right|^2\dx\ds+r\int_{0}^{t}\|\eta(s)\|^2_{\L^2}\ds\\
	 = \dis \frac{1}{2}\|\rho^0\|^2_{\L^2}+ \int_{0}^{t}e^{-rs}\sum_{i=1}^{n} \left\langle f_i(s),\eta_i(s)\right\rangle \ds -\sum_{i=1}^{N}\int_{0}^{t}\int_{0}^{l_i}u_i(s)\eta_i(s)\frac{\partial \eta_i}{\partial x}(s)\dx\ds.
	\end{array}
	\end{equation}
	Using Young's inequality, we can write
	\begin{equation*}
	\begin{array}{rlll}
	\dis \int_{0}^{t}e^{-rs}\sum_{i=1}^{N} \left\langle f_i(s),\eta_i(s)\right\rangle \ds  &\leq& \dis \frac{1}{\underline{D}}\|f\|^2_{L^2(0,T;\V^\star)} +\frac{\underline{D}}{4}\int_{0}^{t}\|\eta(s)\|^2_{\V}\ds,\\
	\dis	\left| \sum_{i=1}^{N}\int_{0}^{t}\int_{0}^{l_i} u_i\eta_i\frac{\partial \eta_i}{\partial x} \, \dx\ds\right|&\leq& \dis \frac{\underline{D}}{4}\int_{0}^{t}\left\|\frac{\partial \eta}{\partial x}(s)\right\|^2_{\L^2}ds+\frac{\|u\|^2_{\L^\infty}}{\underline{D}}\int_{0}^{t}\|\eta(s)\|^2_{\L^2}\ds.
	\end{array}
	\end{equation*}
	Therefore, using these latter estimations and choosing $r=\frac{3\underline{D}}{4}+\frac{\|u\|^2_{\L^\infty}}{\underline{D}}$ in \eqref{eq0}, we get
	\begin{equation*}
	\begin{array}{lll}
	\dis \frac{1}{2}\|\eta(t)\|^2_{\L^2} + \frac{\underline{D}}{2}\int_{0}^{t}\|\eta(s)\|^2_{\V}\, \ds\leq \dis \frac{1}{2}\|\rho_0\|^2_{\L^2}+ \frac{1}{\underline{D}}\|f\|^2_{L^2(0,T;\V^\star)}.
	\end{array}
	\end{equation*}	
\vspace{-1cm}
	Hence, 
	\begin{equation}\label{eq1}
	\begin{array}{lll}
	\dis \|\eta\|_{\mathcal{C}([0,T];\L^2)} \leq \sqrt{1+\frac{2}{\underbar{D}}} \left(\|\rho^0\|_{\L^2}+\|f\|_{L^2(0,T;\V^\star)}\right),\vspace{0.2cm}\\
	
	\|\eta\|_{L^2(0,T;\V)} \leq \sqrt{\frac{1}{\underbar{D}}+\frac{2}{\underbar{D}^2}}\left(\|\rho^0\|_{\L^2}+\|f\|_{L^2(0,T;\V^\star)}\right).
	\end{array}
	\end{equation}	
	Transforming back by $\eta=e^{-rt}\rho$ lead us to \eqref{estimationgen}. 
	By classical straightforward computations on the first equation of \eqref{general} while using \eqref{estimationgen}, we obtain \eqref{estimationge}.
	
	\noindent \textbf{Step 3.} We prove the uniqueness of $\rho$ solution to \eqref{general}.\\
	Assume that there exist $\rho_1$ and $\rho_2$ solutions to \eqref{general} with the same data.  Set $\eta:=\rho_1-\rho_2$.
	Then $\eta\in \W(0,T)$ is solution to
	\begin{equation}\label{uniqueness}
	\left\{
	\begin{array}{lllllllllllllllll}
	\displaystyle	\frac{\partial \eta_i}{\partial t}-D_i\frac{\partial^2 \eta_i}{\partial x^2}-\frac{\partial }{\partial x}\left(u_i\eta_i\right)&=&0&\text{in}& Q_i,\\[6pt]
	\displaystyle	\eta_i(x,0)&=&0,& &x\in(0,l_i),\\[3pt]
	\displaystyle	\sum_{i=1}^N\left[D_i\frac{\partial \eta_i}{\partial x}(0,t)+u_i(0,t)\eta_i(0,t)\right]&=&0,& & \\[3pt]
	\displaystyle \eta_i(0,t)&=&\eta_j(0,t), & i\neq j,& t\in (0,T),\\ [3pt]
	\displaystyle	\eta_i(l_i,t)&=&0.& &
	\end{array}
	\right.
	\end{equation}
	Next, we test the weak formulation of \eqref{uniqueness} with $\varphi=\eta(t)$, to obtain:
	\begin{equation}\label{eq0u}
	\begin{array}{lll}
	\dis \frac{1}{2} \frac{d}{dt}\|\eta(t)\|^2_{\L^2} + \underbar{D}\left\|\frac{\partial \eta}{\partial x}(t) \right\|^2_{\L^2} \leq - \dis\sum_{i=1}^{N}\int_{0}^{l_i} u_i(t)\eta_i(t)\frac{\partial \eta_i}{\partial x}(t) \, \dx.
	\end{array}
	\end{equation}
	Using Young's inequality, we obtain
	\begin{equation*}
	\begin{array}{rlll}
	- \dis\sum_{i=1}^{N}\int_{0}^{l_i} u_i(t)\eta_i(t)\frac{\partial \eta_i}{\partial x}(t) \, \dx \leq \dis \frac{\underbar{D}}{2}\left\|\frac{\partial \eta}{\partial x}(t) \right\|^2_{\L^2}+ \frac{1}{2\underbar{D}}\|u\|^2_{\L^\infty}\| \eta(t)\|^2_{\L^2}.
	\end{array}
	\end{equation*}
	Therefore, using this latter estimation in \eqref{eq0u}, we deduce that
	\begin{equation}\label{eq1u}
	\begin{array}{lll}
	\dis  \frac{d}{dt}\|\eta(t)\|^2_{\L^2} \leq \frac{1}{\underbar{D}} \|u\|^2_{\L^\infty}\| \eta(t)\|^2_{\L^2}. 
	\end{array}
	\end{equation}
	Using the Gronwall inequality and since $\eta(\cdot,0)=0$, we arrive to $\eta=0$ a.e. in $Q$. Hence $\rho_1=\rho_2$. This completes the proof.
\end{proof}

By taking $f=0$ in \eqref{general}, we deduce the following well-posedness result for \eqref{model}.
\begin{theorem}\label{existencemodel}
	Let $u\in \L^\infty$  and $\rho^{0}\in \L^2$.  Then, there exists a unique weak solution $\rho=(\rho_i)_i\in \W(0,T)$ to \eqref{model}.
	In addition, there exists a constant $C>0$ depending on $T$, $\underbar{D}$, $\bar{D}$ and continuously on $\|u\|_{\L^\infty}$ such that
	\begin{align}
	\label{estmodel}
	\|\rho\|_{\mathcal{C}([0,T];L^2(\Omega))}+ \|\rho\|_{L^2(0,T;\V)} &\leq C\|\rho^0\|_{\L^2},
	\\
	\label{estimmodelw}
	\|\rho\|_{\W(0,T)}&\leq  C\|\rho^0\|_{\L^2}.
	\end{align}
\end{theorem}

\begin{remark}
Let $u = (u_i)_{i=1}^N \in \mathbb{L}^\infty$ and $\rho^0 = (\rho_i^0)_{i=1}^N \in \mathbb{L}^2$, and let 
$\rho = (\rho_i)_{i=1}^N$ be the weak solution of \eqref{model}. Then we observe the following:
\begin{enumerate}
\item For every $t \in [0,T]$, if we integrate \eqref{model}, then  obtain
\begin{equation*}
	\sum_{i=1}^N \int_0^{l_i} \rho_i(x,t)\,dx
	=\sum_{i=1}^N \int_0^{l_i} \rho_i^0(x)\,dx -\sum_{i=1}^N \int_0^t \bigl(D_i \partial_x \rho_i + u_i \rho_i\bigr)\big|_{x=l_i}(\tau)\,d\tau.
\end{equation*}
In particular, if the outer boundary conditions at $x=l_i$ are chosen so that
$$
\left(D_i \partial_x \rho_i + u_i \rho_i\right)\big|_{x=l_i} = 0
\quad \text{for all } t\in[0,T],\; i=1,\dots,N,
$$
then the total mass is conserved:
\begin{equation}\label{mass}
\sum_{i=1}^N \int_0^{l_i} \rho_i(x,t)\,dx=\sum_{i=1}^N \int_0^{l_i} \rho_i^0(x)\,dx
\quad \text{for all } t\in[0,T].
\end{equation}

\item If $\rho_i^0(x) \ge 0$ a.e.\ in $(0,l_i)$ for all $i=1,\dots,N$, then we can prove that
$$
\rho_i(x,t) \ge 0 \quad \text{for a.e. } (x,t) \in (0,l_i)\times(0,T), \quad i=1,\dots,N.
$$
\end{enumerate}
\end{remark}

\section{Existence of optimal solutions and first-order optimality conditions}\label{exopt}
In this section, we prove the existence of optimal solutions to \eqref{opt}-\eqref{defuad}.
Let us introduce the control to state operator $G:\L^{\infty} \rightarrow \W(0,T)$ which associates to each $u\in \L^\infty$, the unique solution $\rho$ of \eqref{model}. Then, according to Theorem \ref{existencemodel}, $G$ is well-defined. 
We recall that 
\begin{equation*}
	\begin{array}{lll}
		J(\rho,u)&:=&\dis \frac{1}{2}\sum_{i=1}^{N}\int_{Q_i} \abs{\rho_i-\rho^d_i}^2\,\dq+\frac{1}{2}\sum_{i=1}^{N}\int_{0}^{l_i} \abs{\rho_i(T)-\rho^T_i}^2\,\dx\\
		&&\dis +\frac{1}{2}\sum_{i=1}^{N}\alpha_i\int_{Q_i}
		|u_i|^{2} \,\dq.
	\end{array}
\end{equation*}
\begin{theorem}\label{existenceofcontrols}
	Let $\rho^d=(\rho_i^d)_i \in L^2(0,T;\L^2)$, $\rho^T=(\rho_i^T)_i \in \L^2$ and  $\alpha_i>0$ for $i=1,2,\ldots,N$. Then  there exists at least a solution  $u\in \U$ of the optimal control problem \eqref{opt}-\eqref{defuad}.
\end{theorem}
\begin{proof}
We first note that $J(G(u),u)\geq 0$ for all $u\in \mathcal{U}_{ad}$. Let $\{(G(u_k),u_k)\}_k\subset \W(0,T)\times\mathcal{U}_{ad}$ be a minimizing sequence such that
	$$\lim_{k\to \infty}J(G(u_k),u_k)= \inf_{u\in\U}J(G(u),u). $$	
	Then, there exists a constant $C>0$ independent of $k$ such that
	\begin{equation}\label{boundcontrol}
	\|u_k\|_{L^2(0,T;\L^2)}\leq C.
	\end{equation}
Next, we have that $\rho_k:=G(u_k)$ is the solution of the following problem:
	\begin{equation}\label{modelk}
	\left\{
	\begin{array}{lllllllllllllllll}
	\displaystyle	\frac{\partial (\rho_i)_k}{\partial t}-D_i\frac{\partial^2 (\rho_i)_k}{\partial x^2}-\frac{\partial }{\partial x}\left((u_i)_k(\rho_i)_k\right)&=&0&\text{in}& Q_i,\\[6pt]
	\displaystyle	(\rho_i)_k(x,0)&=&\rho^0_i(x),& &x\in(0,l_i),\\[3pt]
	\displaystyle	\sum_{i=1}^N\left[D_i\frac{\partial (\rho_i)_k}{\partial x}(0,t)+\left((u_i)_k(\rho_i)_k\right)(0,t)\right]&=&0,& & \\[3pt]
	\displaystyle (\rho_i)_k(0,t)&=&(\rho_j)_k(0,t), & i\neq j,&  t\in (0,T),\\ [3pt]
	\displaystyle	(\rho_i)_k(l_i,t)&=&0.& & 
	\end{array}
	\right.
	\end{equation}
	
Then, it follows from \eqref{boundcontrol} and \eqref{estimmodelw} that there is a positive constant $C$ independent of $k$ such that
	\begin{equation}\label{b2}	
	\|\rho_k\|_{\W(0,T)} \leq C.
	\end{equation}
	From the boundedness of $\U$ in $\L^\infty$, \eqref{boundcontrol} and \eqref{b2}, we deduce the existence of  $\bar{u}\in \L^\infty$ and $\bar{\rho}\in \W(0,T)$ such that up to a subsequence and as $k\to \infty$, we have
	\begin{equation}\label{c1}
	u_k\rightharpoonup \bar{u}   \text{ weakly-$\star$ in } L^\infty(0,T;\L^\infty),
	\end{equation}	
    \vspace{-0.8cm}
	\begin{equation}\label{c1a}
	u_k\rightharpoonup \bar{u}   \text{ weakly in } L^2(0,T;\L^2),
	\end{equation}
 \vspace{-0.8cm}
	\begin{equation}\label{c2}
	\rho_k\rightharpoonup \bar{\rho}   \text{ weakly in }  \W(0,T),
	\end{equation}
	and  for each $t\in [0,T]$,
		\begin{equation}\label{c2p}
	\rho_k(t)\rightharpoonup \bar{\rho}(t)   \text{ weakly in } \L^2.
	\end{equation}
	Using the compact embedding $\W(0,T)\hookrightarrow L^2(0,T;\L^2)$, we deduce from \eqref{c2} that
	\begin{equation}\label{c3}
	\rho_k\to \bar{\rho}   \text{ strongly in }  L^2(0,T;\L^2).
	\end{equation}

	Finally, we pass to the limit in \eqref{modelk} using the above convergence results and classical arguments to obtain that $\bar{\rho}=\bar{\rho}(\bar{u})$, i.e. $\bar{\rho}$ is the state associated to the control $\bar{u}$.  
	In addition, since $\U$ is a closed convex subset of $L^2(0,T;\L^2)$, we have that $\U$ is weakly closed and \eqref{c1a} implies 
\begin{equation}\label{weakclosure}
\bar{u}\in \U.
\end{equation}
	Moreover, using \eqref{c2p}, \eqref{c3}, \eqref{weakclosure}, and the weak lower semi-continuity of the cost functional $J$, it follows that
	\begin{eqnarray*}
	J(G(\bar{u}),\bar{u})&\leq &\liminf_{k\to \infty}J(G(u_k),u_k)=\inf_{u\in\U}J(G(u),u).
	\end{eqnarray*}	
	This completes the proof.		
\end{proof}

\begin{remark}\label{nonunique} 
	The objective functional associated with the optimization problem \eqref{opt}-\eqref{defuad} is non-convex. As a result, the optimal solution may not be unique. Therefore, we focus on local optimal solutions, as defined below.
\end{remark}

\begin{definition}\label{defopt}
We say that $\bar{u}\in \U$ is a $L^2$-local solution of \eqref{opt} if there exists $\varepsilon>0$ such that $J(G(\bar{u}),\bar{u})\leq J(G(u),u)$ for every $u\in \U\cap B^2_\varepsilon(\bar{u})$ where $ B^2_\varepsilon(\bar{u})=\{u\in L^2(0,T;\L^2):\|u-\bar{u}\|_{L^2(0,T;\L^2)}\leq \varepsilon\}$. We say that $\bar{u}$ is a strict local minimum of \eqref{opt} if the above inequality is strict whenever $u\neq \bar{u}$.
\end{definition}

We note that the weak formulation of \eqref{model} can be rewritten as follows, 
\begin{equation}\label{weakg}
\begin{array}{lll}
\dis \left\langle\rho_t,\cdot\right\rangle + \mathcal{F}[u](\rho,\cdot) = 0, \;\text{ in }\; L^2(0,T;\V^\star),\\
\dis  \rho(\cdot,0)=\rho^0 \text{ in } \L^2,
\end{array}
\end{equation}
where $\mathcal{F}[u]$ is defined in \eqref{defFt}

Before proceeding further, we need to establish some regularity results for the control-to-state operator.
\subsection{Regularity results for the control-to-state operator}
Let us introduce the mapping
\begin{equation}\label{defG}
(\rho, u)\mapsto\mathcal{G}(\rho,u):= (\dis \rho_{t}+\mathcal{F}[u](\rho,\cdot),\rho(0)-\rho^0)
\end{equation}
from $ \W(0,T)\times \L^\infty\to L^2(0,T;\V^\star)\times \L^2,$ with $\mathcal{F}$ defined in \eqref{defFt}. The mapping $\mathcal{G}$ is well defined and the state equation $\rho$ solution of \eqref{model} can be viewed as the equation
\begin{equation}\label{eqG}
\mathcal{G}(\rho,u)=0.
\end{equation}
We have the following result. 
\begin{lemma}\label{lemmeG}
The mapping $\mathcal{G}$ is of class $\mathcal{C}^{\infty}$. Moreover, the control-to-state mapping $G:\L^\infty\to \W(0,T), u\mapsto \rho$ solution of \eqref{model} is also of class $\mathcal{C}^{\infty}$.
\end{lemma}
\begin{proof}
	The first part of the lemma is straightforward. To prove the second part, we use the Implicit Function Theorem. We have, for any $\varphi\in \W(0,T)$, 
	$$\partial_\rho\mathcal{G}(\rho,u)\varphi=(\dis \varphi_{t}+\mathcal{F}[u](\varphi,\rho),\varphi(0)).$$
	Since $\mathcal{F}[u](\cdot,\cdot)$ is continuous and weakly coercive (see proof of Theorem \ref{theoremexistence}), we deduce from \cite[Theorem 11.7]{chipot2000} that for any $\varphi^0\in \L^2$ and $f\in L^2(0,T; \V^\star)$, the linear problem
	\begin{equation}\label{weakgphi}
	\begin{array}{lll}
	\dis \left\langle\varphi_t,\cdot\right\rangle + \mathcal{F}[u](\varphi,\cdot) = f, \;\text{ in }\; L^2(0,T;\V^\star),\\
	\dis  \varphi(\cdot,0)=\varphi^0 \text{ in } \L^2,
	\end{array}
	\end{equation}
	admits a unique weak solution  $\varphi:=\varphi(\varphi^0,f)$ in $\W(0,T)$ depending continuously on $\varphi^0\in \L^2$ and on $f\in L^2(0,T; \V^\star)$. Hence,
	$\dis \partial_\rho\mathcal{G}(\rho,u)$ defines an isomorphism from $\W(0,T)$ to $L^2(0,T; \V^\star)\times \L^2$. The Implicit Function Theorem applies, and we deduce that $\mathcal{G}(\rho,u)=0$ implicitly defines the control-to-state operator $G:u\mapsto \rho$ which is itself of class  $\mathcal{C}^{2}$.
\end{proof}
The next result gives the first derivative of the control-to-state operator. Let $w\in \L^\infty$. We introduce the functional $d[w]\in \V^\star$ defined by:
\begin{equation}\label{defd}
d[w](\varphi)=-\sum_{i=1}^{N}\int_{0}^{l_i}w_i\frac{\partial \varphi_i}{\partial x}\,\dx.
\end{equation}
Next, we have the following result.
\begin{proposition}\label{differentiability}
	Let $u,v\in \L^\infty$.	Then, the first-order derivative of the control-to-state operator $G'(u)$ is given by $G'(u)v=z_v$, where $z_v\in \W(0,T)$ is the unique solution to
	\begin{equation}\label{diff1}
	\begin{array}{lll}
	\dis \left\langle z_t,\cdot\right\rangle + \mathcal{F}[u](z,\cdot) = d[\rho v], \;\text{ in }\; L^2(0,T;\V^\star),\\
	\dis  z(\cdot,0)=0 \text{ in } \L^2,
	\end{array}
	\end{equation}
	where $\mathcal{F}[u]$ is defined in \eqref{defFt}.
\end{proposition}

\begin{proof}
	Let $u,v\in \L^\infty$. The derivative $G'(u)v$ exists according to Lemma \ref{lemmeG} and \eqref{diff1} can be easily derived through straightforward computations. Additionally, note that the system \eqref{diff1} can be written in a form similar to the system \eqref{weakgene}, where  $f=d[\rho v]\in L^2(0,T;\V^\star)$. Therefore, we can apply Theorem \ref{theoremexistence}  to deduce that \eqref{diff1} admits a unique solution $z_v\in \W(0,T)$. This ends the proof.
\end{proof}
\begin{remark}\label{strong formulation}
	We observe that problem \eqref{diff1} can be seen as the weak formulation of the following problem:
	\begin{equation}\label{diff1strong}
	\left\{
	\begin{array}{lllllllllllllllll}
	\displaystyle	\frac{\partial z_i}{\partial t}-D_i\frac{\partial^2 z_i}{\partial x^2}-\frac{\partial }{\partial x}\left(u_iz_i\right)&=&\dis \frac{\partial}{\partial x}(\rho_iv_i)&\text{in}& Q_i,\\[6pt]
	\displaystyle	z_i(x,0)&=&0& &x\in(0,l_i),\\[3pt]
	\displaystyle	\sum_{i=1}^N\left[D_i\frac{\partial z_i}{\partial x}(0,t)+u_i(0,t)z_i(0,t)\right]&=&-\rho_i(0,t)v_i(0,t),& & \\[3pt]
	\displaystyle z_i(0,t)&=&z_j(0,t), & i\neq j,&  t\in (0,T),\\ [3pt]
	\displaystyle	z_i(l_i,t)&=&0.& & 
	\end{array}
	\right.
	\end{equation}
\end{remark}

\subsection{The adjoint equation and first-order necessary optimality conditions}
Next, we define the adjoint state $q$ as the unique weak solution to the adjoint equation:
\begin{equation}\label{adjoint}
\left\{
\begin{array}{lllllllllllllllll}
\displaystyle -	\frac{\partial q_i}{\partial t}-D_i\frac{\partial^2 q_i}{\partial x^2}+u_i\frac{\partial q_i }{\partial x}&=&\rho_i-\rho^d_i&\text{in}& Q_i,\\[6pt]
\displaystyle	q_i(x,T)&=&\rho_i(x,T)-\rho^T_i(x)& &x\in(0,l_i),\\[3pt]
\displaystyle	\sum_{i=1}^ND_i\frac{\partial q_i}{\partial x}(0,t)&=&0,& & \\[3pt]
\displaystyle q_i(0,t)&=&q_j(0,t), & i\neq j,& t\in (0,T),\\ [3pt]
\displaystyle	q_i(l_i,t)&=&0.& & 
\end{array}
\right.
\end{equation}

We now give a definition of weak solution to the adjoint equation \eqref{adjoint}.
\begin{definition}\label{weakadjoint}
	Let $\rho^d=(\rho^d_i)\in L^2(0,T;\L^2)$, $\rho^T=(\rho^T_i)\in \L^2$  and $u\in \L^\infty$.   We  say that a function
	$q\in \W(0,T)$ is a weak solution to \eqref{adjoint}, if:
	\begin{equation}\label{weakad}
	\begin{array}{lll}
	\dis- \left\langle q_t,\cdot\right\rangle + \mathcal{F}[u](\cdot,q) = \left\langle \rho-\rho^d,\cdot\right\rangle \;\text{ in }\; L^2(0,T;\V^\star),\\
	\dis  q(\cdot,T)= \rho(\cdot,T)-\rho^T \text{ in } \L^2,
	\end{array}
	\end{equation}
	where $\mathcal{F}[u]$ is defined in \eqref{defFt}.
\end{definition}
\begin{proposition}\label{existenceadjoint}
Let	 $\rho^d=(\rho^d_i)\in L^2(0,T;\L^2)$, $\rho^T=(\rho^T_i)\in \L^2$  and $u\in \L^\infty$. Then, there exists a unique solution $q\in \W(0,T)$ to \eqref{weakad}.  Moreover, there exists a constant $C>0$ depending on $T$, $\underbar{D}$, and continuously on $\|u\|_{\L^\infty}$ such that
	 \begin{equation}\label{estadjoint}
	 \|q\|_{\W(0,T)}\leq  C\left(\|\rho^0\|_{\L^2}+\|\rho^d\|_{L^2(0,T;\L^2)}+\|\rho^T\|_{\L^2}\right).
	 \end{equation}
\end{proposition}
\begin{proof} We make a change of variable $t\mapsto q(T-t,x)$ and using the same arguments as in Theorem \ref{theoremexistence}, we can deduce the existence of a unique solution $q\in \W(0,T)$ to \eqref{weakad} and the estimate \eqref{estadjoint} follows from \eqref{estimmodelw}.
\end{proof}
\begin{remark}
	We note that when $u \in \mathcal{U}_{ad}$ is taken in Theorems \ref{theoremexistence}, \ref{existencemodel} and Proposition \ref{existenceadjoint}, the constant $C$ no longer depends on $\|u\|_{\L^\infty}$ but rather on the $\L^\infty $-bounds of $u^{\min}$ and $u^{\max}$. 
\end{remark}

Let us introduce the reduced cost functional as follows: 
\begin{equation}\label{defj}
\mathcal{J}(u):=J(G(u),u).
\end{equation}
The next proposition gives regularity results for $\mathcal{J}$.
\begin{proposition}\label{diff4}
	Let $\rho$ be the solution of \eqref{model}.  Then, the functional $\mathcal{J}:\L^\infty\to \R$ defined in \eqref{defj} is continuously Fr\'echet differentiable and for every $u, v\in \L^\infty$, we have
	\begin{equation}\label{diff5}
	\begin{array}{lllll}
	\mathcal{J}'(u)v &=& \dis \sum_{i=1}^{N}\int_{Q_i}\left(\alpha_i u_i-\rho_i\frac{\partial q_i}{\partial x}\right)v_i\, \dq,
	\end{array}
	\end{equation}
	where $q$ is the unique weak solution to the adjoint equation \eqref{adjoint}.
\end{proposition}
\begin{proof}
	First, by the chain rule, we find that $\mathcal{J}$ is continuously Fr\'echet differentiable because, as stated in Lemma \ref{lemmeG}, $G$ possesses this property.\par 		
	Let $u,v\in \L^\infty$.  After performing some straightforward classical computations, we obtain 
	\begin{equation}\label{e1}
	\mathcal{J}'(u)v=\sum_{i=1}^{N}\int_{Q_i}z_i(\rho_i-\rho_i^d)\,\dq+ \sum_{i=1}^{N}\int_{Q_i}z_i(T)(\rho_i(T)-\rho_i^T)\,\dx +\sum_{i=1}^{N}\alpha_i\int_{Q_i}u_iv_i\,\dq.
	\end{equation}
	Now, if we test the first equation of \eqref{diff1} by $q$ solution to the adjoint state \eqref{adjoint} and we integrate by parts over $(0,T)$, we arrive to
	\begin{equation}\label{e2}
\sum_{i=1}^{N}\int_{Q_i}z_i(\rho_i-\rho_i^d)\,\dq+ \sum_{i=1}^{N}\int_{Q_i}z_i(T)(\rho_i(T)-\rho_i^T)\,\dx =-\sum_{i=1}^{N}\int_{Q_i}\rho_iv_i\frac{\partial q_i}{\partial x}\,\dq.
	\end{equation}
	Combining \eqref{e1} and \eqref{e2}, we arrive to \eqref{diff5}. 
\end{proof}
Using the above results, we derive the first-order necessary conditions for optimality and characterize the optimal solutions. The proof is based on classical methods and relies on the convexity of the set of admissible controls.
\begin{theorem}\label{theoSO}
	Let $\bar{u}\in\mathcal U_{ad}$ be a $L^2$-local solution of \eqref{opt}. Then,
	\begin{equation}\label{ineq}
	\mathcal{J}'(\bar{u})(v-\bar{u})\geq 0\;\;\;\text{for every}\;\;\; v\in \U, \quad\text{equivalently}
	\end{equation}
    \vspace{-0.8cm}
	\begin{equation}\label{ineq1}
	\dis \sum_{i=1}^{N}\int_{Q_i}\left(\alpha_i \bar{u}_i-\rho_i\frac{\partial q_i}{\partial x}\right)(v_i-\bar{u}_i)\, \dq\geq 0\;\;\;\text{for every}\;\;\; v\in \U,
	\end{equation}
	where $q$ is the unique weak solution to \eqref{adjoint} and $\rho$ is the unique weak solution to \eqref{model}.
\end{theorem}
The optimality condition derived in \eqref{ineq1} can be written as a Min-max model \cite{manzoni2021optimal}.
Having established the optimality conditions for the control problem \eqref{opt}-\eqref{defuad}, we now proceed to develop a wavelet-based numerical method aimed at approximating the solution of the associated system of necessary optimality conditions, as defined by equations \eqref{model}, \eqref{adjoint}, and \eqref{ineq1}.

\section{Numerical Illustration}\label{numerical illus}
\subsection{Shifted Legendre Wavelets}
We employ the shifted Legendre wavelet collocation method \cite{Kumari11062025} to approximate the solution of the necessary optimality conditions derived above. Wavelet-based methods have been shown to outperform classical orthogonal polynomial techniques in approximating non-smooth functions, owing to its advantageous properties such as compact support, vanishing moments, and multiresolution analysis. For a general overview of wavelets, we refer the reader to \cite{mehra2018wavelets}. The shifted Legendre scaling functions
are defined on the interval $[0,1)$ as follows:
\begin{align}
    \phi_{n,m}^J(t)=\begin{cases}
        2^\frac{J-1}{2} \sqrt{2m + 1} L_m(2^{J-1}t - n + 1) \ \ &\text{if} \ \frac{n-1}{2^{J-1}} \leq t < \frac{n}{2^{J-1}},\\
        0 &\text{otherwise},
    \end{cases}
\end{align}
where $n = 1,2,\ldots2^{J-1}$ is the translation parameter, $J=1,2,\ldots$ is the dilation parameter and $m=0,1,2,\ldots,M-1$ is the degree of the shifted Legendre polynomials \cite{bhrawy2012shifted}, defined as
$$L_m(t)=\sum_{i=0}^m c_{i,m} t^{i},\quad
\text{where}\quad
c_{0,0}=1, \ \ c_{i,m}=\frac{\prod_{j=0}^{m-1}(1+i+j)}{\prod_{\substack{j=0\\j\neq i}}^m(i-j)}.$$

 The shifted Legendre scaling function expansion of an arbitrary function $f(t) \in L_2(0,1)$ can be expressed as
\begin{align*}
    f(t)= \sum_{n=1}^\infty \sum_{m=0}^{M-1}f_{n,m}^J\phi_{n,m}^J(t),
\end{align*}
where the coefficients are given by $f_{n,m}^J = \int_0^1 f(t)\phi_{n,m}^J(t)\,dt$, and $J \to \infty$. However, for practical numerical implementation using the collocation method, we consider a truncated form of this expansion, given by

\begin{align}
    f(t) \approx \sum_{n=1}^{2^{J-1}} \sum_{m=0}^{M-1}f_{n,m}^J\phi_{n,m}^J(t) = F^T\Phi_{2^{J-1}M}(t),
\end{align}
such that it is the best approximation of $f(t)$ in $L_2$ sense, where $F$ and $\Phi_{2^{J-1}M}(t)$ are vectors of length $2^{J-1}M$, defined respectively as
\begin{equation}
\begin{aligned}\label{span}
F=&[f_{1,0}^J,f_{1,1}^J,\ldots,f_{1,M-1}^J,f_{2,0}^J,\ldots,f_{2,M-1}^J,\ldots,f_{2^{J-1},0}^J,\ldots,f_{2^{J-1},M-1}^J]^T,\\
\Phi_K(t)=&\Phi_{2^{J-1}M}(t)=[\phi_{1,0}^J(t),\phi_{1,1}^J(t)\ldots,\phi_{2^{J-1},0}^J(t),\ldots,\phi_{2^{J-1},M-1}^J(t)]^T \\
     =& [\phi_{1}^J(t),\phi_{2}^J(t),\ldots,\phi_{M}^J(t),\ldots,\phi_{K}^J(t)]^T,
\end{aligned}
\end{equation}
where $K=2^{J-1}M$.
Similarly, a function in two variables $f(x,t)$ can be expanded as
\begin{align}\label{funcapprox}
  f(x,t) \approx \sum_{n_1=1}^{2^{J_1-1}} \sum_{m_1=0}^{M_1-1}\sum_{n_2=1}^{2^{J_2-1}} \sum_{m_2=0}^{M_2-1}f_{n_1,m_1}^{n_2,m_2}\phi_{n_1,m_1}^{J_1}(x)\phi_{n_2,m_2}^{J_2}(t) = (\Phi_{K_1}(x))^T F\Phi_{K_2}(t),   
\end{align}
such that it is the best approximation of $f(x,t)$ in $L_2$ sense, where $K_1=2^{J_1-1}M_1$, $K_2=2^{J_2-1}M_2$ and $F$ is a constant matrix of order $K_1\times K_2$. 

To numerically approximate the state, adjoint, and control variables, it is necessary to evaluate integrals of the shifted Legendre scaling function basis defined in \eqref{span}. Accordingly, we provide exact expressions for the left- and right-sided single and double integrals of the function \( \Phi_{2^{J-1}M}(t) \), also denoted by \( \Phi_K(t) \).

\begin{result}
The exact left integral $({}_0 I_x)$ of the shifted Legendre scaling function basis $\Phi_K(t)$ is denoted and defined as
\begin{equation}
       {}_0 I_x\left(\Phi_{K} (t)\right)=[f_k^{(1)}(t)]=\begin{pmatrix}
            f_1^{(1)}(t)&
            f_2^{(1)}(t)&
            \ldots&
            f_k^{(1)}(t)&
            \ldots&
            f_K^{(1)}(t)         
             \end{pmatrix}^T,
             \end{equation}
            where the index $k=(n-1)M+m$ for $n=1,2,\ldots,2^{J-1}$, $m=1,2,\ldots,M$ such that
\begin{equation*}
 f_k^{(1)}(t)=\begin{cases}
             \sqrt{\frac{2m+1}{2^{J-1}}}\sum_{i=0}^m\left(\frac{c_{i,m}}{i+1}(2^{J-1}x-n+1)^{i+1}\right)&\text{if}\quad \frac{n-1}{2^{J-1}}<t<\frac{n}{2^{J-1}},\\ \sqrt{\frac{1}{2^{J-1}}} &\text{if}\quad t\geq \frac{n}{2^{J-1}}\quad \text{and} \quad m=0,\\
            0 &\text{otherwise},
             \end{cases}
    \end{equation*}
    for $k=1,2,\ldots,K.$
\end{result}
\begin{result}
The exact right integral $({}_xI_1)$ of the shifted Legendre scaling function basis $\Phi_K(t)$ is denoted and defined as
\begin{equation}
        {}_x I_1\left(\Phi_{K} (t)\right)=[f_k^{(2)}(t)]=\begin{pmatrix}
            f_1^{(2)}(t)&
            f_2^{(2)}(t)&
            \ldots&
            f_k^{(2)}(t)&
            \ldots&
            f_K^{(2)}(t)         
             \end{pmatrix}^T,
             \end{equation}
            where the index $k=(n-1)M+m$ for $n=1,2,\ldots,2^{J-1}$, $m=1,2,\ldots,M$ such that
\begin{equation*}
 f_k^{(2)}(t)=\begin{cases}
            \sqrt{\frac{1}{2^{J-1}}}&\text{if}\quad t\leq \frac{n-1}{2^{J-1}}\  \text{and} \ m=0,\\ 
             \sqrt{\frac{2m+1}{2^{J-1}}}\sum_{i=0}^m\left(\frac{c_{i,m}}{i+1}\left(1-(2^{J-1}x-n+1)^{i+1}\right)\right)&\text{if}\quad \frac{n-1}{2^{J-1}}<t<\frac{n}{2^{J-1}},\\ 0 &\text{otherwise},
             \end{cases}
    \end{equation*}
    for $k=1,2,\ldots,K.$
\end{result}
\begin{result}
The exact left double integral $({}_0 I_x^2)$ of the shifted Legendre scaling function basis $\Phi_K(t)$ is denoted and defined as
\begin{equation}
        {}_0 I_x^2\left(\Phi_{K} (t)\right)=[f_k^{(3)}(t)]=\begin{pmatrix}
            f_1^{(3)}(t)&
            f_2^{(3)}(t)&
            \ldots&
            f_k^{(3)}(t)&
            \ldots&
            f_K^{(3)}(t)         
             \end{pmatrix}^T,
             \end{equation}
            where the index $k=(n-1)M+m$ for $n=1,2,\ldots,2^{J-1}$, $m=1,2,\ldots,M$ such 
\begin{equation*}
 \!\!\!f_k^{(3)}(t)=\begin{cases}
            0&\text{if}\quad t\leq \frac{n-1}{2^{J-1}},\\ \sqrt{\frac{2m+1}{(2^{J-1})^3}}\sum_{i=0}^m\left(\frac{c_{i,m}}{(i+1)(i+2)}(2^{J-1}x-n+1)^{i+2}\right)&\text{if}\quad \frac{n-1}{2^{J-1}}<t<\frac{n}{2^{J-1}},\\ \sqrt{\frac{1}{2^{J-1}}}\left(\frac{2^{J-1}t-n+1}{2^{J-1}}-\frac{1}{2^J}\right) &\text{if}\quad t\geq \frac{n}{2^{J-1}} \ \text{and} \  m=0,\\
            \sqrt{\frac{2m+1}{2^{J-1}}}\left(-\frac{1}{2^{J-1}}\sum_{i=0}^m\frac{c_{i,m}}{i+2}\right) &\text{if}\quad t\geq \frac{n}{2^{J-1}} \  \text{and} \  m\neq0,
             \end{cases}
    \end{equation*}
    for $k=1,2,\ldots,K.$
\end{result}
In the following subsection, we utilize the results derived above to construct a numerical scheme for solving the FPOCP.

\subsection{Numerical Scheme}
Now, we construct the numerical scheme using the shifted Legendre scaling function basis defined in \eqref{span} for the solution of the optimality system \eqref{model}, \eqref{adjoint}, and \eqref{ineq1}. In this process, the mixed highest derivative of the state variables is approximated as follows:
\begin{equation}\label{triple}
    \frac{\partial^3 \rho_i(x,t)}{\partial x^2 \partial t}\approx\left(\Phi_{K_1}(x)\right)^T A_i \Phi_{K_2}(t), \quad i=1,2,\ldots N.
\end{equation}
Integrating this expression with respect to $t$ in the interval $(0,t)$ and using the initial condition $\rho_i(x,0)=\rho_i^0(x)$, we get
\begin{equation}\label{double}
    \begin{aligned}
        \frac{\partial^2 \rho_i(x,t)}{\partial x^2}&\approx \left(\Phi_{K_1}(x)\right)^T A_i \ {}_0I_t\Phi_{K_2}(t)+\frac{d^2 \rho_i^0(x)}{dx^2}.
    \end{aligned}
\end{equation}
Now, integrating with respect to $x$ in the interval $(0,x)$, we obtain
\begin{equation}\label{single}
    \begin{aligned}
        \frac{\partial \rho_i(x,t)}{\partial x}\approx{}_0I_x\left(\Phi_{K_1}(x)\right)^T A_i \ {}_0I_t\Phi_{K_2}(t)+\frac{d \rho_i^0(x)}{d x}-\frac{d\rho_i^0(x)}{d x}\bigg|_{x=0}+\frac{\partial \rho_i(x,t)}{\partial x}\bigg|_{x=0}.
    \end{aligned}
\end{equation}
Again, integrating with respect to $x$ in the interval $(0,x)$, we finally obtain
\footnotesize
\begin{equation}\label{state expression}
    \rho_i(x,t)\approx {}_0I_x^2\left(\Phi_{K_1}(x)\right)^T A_i \ {}_0 I_t\Phi_{K_2}(t)+\rho_i^0(x)-\rho_i^0(0)-x\frac{d\rho_i^0(x)}{d x}\bigg|_{x=0}+x\frac{\partial \rho_i(x,t)}{\partial x}\bigg|_{x=0}+\rho_i(0,t).
\end{equation}
\normalsize
Next, we find the expressions for $\frac{\partial \rho_i(x,t)}{\partial x}\bigg|_{x=0}$ and $\rho_i(0,t)$ using the boundary conditions and continuity conditions in \eqref{model}. Putting $x=1$ in \eqref{state expression}, multiplying it with $D_i$, and taking summation over all edges while doing required additions and subtractions, we have
 \begin{align*}
    \sum_{i=1}^N D_i\rho_i(1,t)\approx \sum_{i=1}^N \Bigg[D_i\Bigg({}_0 I_x^2\left(\Phi_{K_1}(1)\right)^T A_i \  {}_0 I_t  \Phi_{K_2}(t)+\rho_i^0(1)-\rho_i^0(0)-\frac{d\rho_i^0(x)}{d x}\bigg|_{x=0}\\+\frac{\partial \rho_i(x,t)}{\partial x}\bigg|_{x=0}+\rho_i(0,t)\Bigg)+\rho_i(0,t)u_i(0,t)-\rho_i(0,t)u_i(0,t)\Bigg].
\end{align*}
Now since, $\rho_i(1,t)=0$ and $\sum_{i=1}^N\left[D_i\frac{\partial \rho_i}{\partial x}(0,t)+u_i(0,t)\rho_i(0,t)\right]=0$, the above equation reduces to
\footnotesize
\begin{align*}
    \sum_{i=1}^N\left( D_i\left( \ {}_0 I_x^2 \left(\Phi_{K_1}(1)\right)^T A_i \  {}_0 I_t \Phi_{K_2}(t)- \rho_i^0(0) -\frac{d\rho_i^0(x)}{d x}\bigg|_{x=0}+ \rho_i(0,t)\right)-\rho_i(0,t)u_i(0,t)\right)\approx 0.
\end{align*}
\normalsize
Finally, using the continuity condition $\rho_i(0,t)=\rho_j(0,t) \quad i\neq j \quad i,j=1,2,\dots,N,$ we obtain
\vspace{-0.6cm}
\small

\begin{equation}\label{rho at 0}
   \rho_i(0,t)\approx\frac{\sum_{j=1}^N\left(D_j\ {}_0 I_x^2\left(\Phi_{K_1}(1)\right)^T A_j \  {}_0 I_t\Phi_{K_2}(t)-D_j\rho_j^0(0)-D_j\frac{d\rho_j^0(x)}{d x}\bigg|_{x=0}\right)}{\sum_{j=1}^{N}\left(u_j(0,t)-D_j\right)}. 
\end{equation}
\normalsize
Again, putting $x=1$ in \eqref{state expression}, we have
\small
\begin{align*}
    \rho_i(1,t)\approx {}_0 I_x^2\left(\Phi_{K_1}(1)\right)^T A_i \Phi_{K_2}(t)+\rho_i^0(1)-\rho_i^0(0)-\frac{d\rho_i^0(x)}{d x}\bigg|_{x=0}+\frac{\partial \rho_i(x,t)}{\partial x}\bigg|_{x=0}+\rho_i(0,t),
\end{align*}
\normalsize
which gives
\begin{equation*}
 \frac{\partial \rho_i(x,t)}{\partial x}\bigg|_{x=0}\approx  -{}_0 I_x^2\left(\Phi_{K_1}(1)\right)^T A_i \ {}_0 I_t\Phi_{K_2}(t)+\rho_i^0(0)+\frac{d\rho_i^0(x)}{d x}\bigg|_{x=0}-\rho_i(0,t). 
\end{equation*}
Thus, the expression for $\frac{\partial \rho_i(x,t)}{\partial x}\bigg|_{x=0}$ using \eqref{rho at 0} is obtained as
\footnotesize
\begin{equation}
\begin{aligned}\label{derivative of rho at 0}
 \frac{\partial \rho_i(x,t)}{\partial x}\bigg|_{x=0}\approx&-{}_0 I_x^2\left(\Phi_{K_1}(1)\right)^T A_i \ {}_0 I_t \Phi_{K_2}(t)+\rho_i^0(0)+\frac{d\rho_i^0(x)}{d x}\bigg|_{x=0}\\&-\frac{\sum_{j=1}^N\left(D_j\ {}_0 I_x^2\left(\Phi_{K_1}(1)\right)^T A_j \  {}_0 I_t\Phi_{K_2}(t)-D_j\rho_j^0(0)-D_j\frac{d\rho_j^0(x)}{d x}\bigg|_{x=0}\right)}{\sum_{j=1}^{N}\left(u_j(0,t)-D_j\right)}. 
\end{aligned}
\end{equation}
\normalsize
Substituting \eqref{rho at 0} and \eqref{derivative of rho at 0} in \eqref{single} and \eqref{state expression} respectively we end up on
\footnotesize
\begin{equation}
\begin{aligned}
\frac{\partial \rho_i(x,t)}{\partial x}\approx{}_0I_x\left(\Phi_{K_1}(x)\right)^T A_i \ {}_0I_t\Phi_{K_2}(t)+\frac{d \rho_i^0(x)}{d x} -{}_0 I_x^2\left(\Phi_{K_1}(1)\right)^T A_i \ {}_0 I_t \Phi_{K_2}(t)+\rho_i^0(0)\\-\frac{\sum_{j=1}^N\left(D_j\ {}_0 I_x^2\left(\Phi_{K_1}(1)\right)^T A_j \  {}_0 I_t\Phi_{K_2}(t)-D_j\rho_j^0(0)-D_j\frac{d\rho_j^0(x)}{d x}\bigg|_{x=0}\right)}{\sum_{j=1}^{N}\left(u_j(0,t)-D_j\right)}, 
 \end{aligned}
\end{equation}
\normalsize
and 
\footnotesize
\begin{equation}\label{final state}
    \begin{aligned}
   \rho_i(x,t)\approx&{}_0I_x^2\left(\Phi_{K_1}(x)\right)^T A_i \  {}_0 I_t\Phi_{K_2}(t)+\rho_i^0(x)-\rho_i^0(0)-x  {}_0I_x^2\left(\Phi_{K_1}(1)\right)^T A_i \ {}_0 I_t\Phi_{K_2}(t)+x\rho_i^0(0)\\&
   +(1-x)  \frac{\sum_{j=1}^N\left(D_j\ {}_0 I_x^2\left(\Phi_{K_1}(1)\right)^T A_j \  {}_0 I_t\Phi_{K_2}(t)-D_j\rho_j^0(0)-D_j\frac{d\rho_j^0(x)}{d x}\bigg|_{x=0}\right)}{\sum_{j=1}^{N}\left(u_j(0,t)-D_j\right)}.  
    \end{aligned}
\end{equation}
\normalsize
Finally taking the derivative of \eqref{final state} with respect to $t$, we conclude
\begin{equation}
    \begin{aligned}
        \frac{\partial\rho_i(x,t)}{\partial t}&\approx{}_0I_x^2\left(\Phi_{K_1}(x)\right)^T A_i \Phi_{K_2}(t)-x  {}_0I_x^2\left(\Phi_{K_1}(1)\right)^T A_i\Phi_{K_2}(t)+(1-x)\\
        &\times \frac{\splitfrac{\Bigg[\left(\sum_{j=1}^{N}\left(u_j(0,t)-D_j\right)\right)\left(\sum_{j=1}^{N}D_j\ {}_0 I_x^2\left(\Phi_{K_1}(1)\right)^T A_j\Phi_{K_2}(t)\right)\mathstrut}{\splitfrac{-\bigg(\sum_{j=1}^N\bigg(D_j\ {}_0 I_x^2\left(\Phi_{K_1}(1)\right)^T A_j \  {}_0 I_t\Phi_{K_2}(t)-D_j\rho_j^0(0)\mathstrut}{-D_j\frac{d\rho_j^0(x)}{d x}\bigg|_{x=0}\bigg)\bigg) \left(\sum_{j=1}^N\frac{d u_j(0,t)}{d t}\right)\Bigg]\quad\mathstrut}}}{\left(\sum_{j=1}^{N}\left(u_j(0,t)-D_j\right)\right)^2}.
         \end{aligned}
 \end{equation}
Similarly, the control and adjoint state variables can also be represented using the shifted Legendre scaling function basis. To proceed, we select \( K_1 \) and \( K_2 \) collocation points for the spatial variable \( x \) and the temporal variable \( t \), respectively. These points are chosen uniformly as
\(
x_{k_1} = \frac{2k_1 - 1}{2^{J_1} M_1}, \quad \text{for } k_1 = 1, 2, \ldots, 2^{J_1 - 1} M_1 = K_1,
\quad \text{and} \quad
t_{k_2} = \frac{2k_2 - 1}{2^{J_2} M_2}, \quad \text{for } k_2 = 1, 2, \ldots, 2^{J_2 - 1} M_2 = K_2.
\)
The expressions obtained for the state, adjoint state, and control variables, along with their derivatives, are then substituted into the corresponding equations \eqref{model}, \eqref{adjoint}, and \eqref{ineq1} at the collocation points \( x_{k_1} \) and \( t_{k_2} \). This results in a system of \( 3N \times K_1 \times K_2 \) algebraic equations, which can be solved to obtain approximate optimal values of the state and control variables.
In the next section, we verify the proposed method for its efficiency to approximate the optimal solutions of the problem for two different examples.
\section{Numerical Results}\label{numerical res}
We consider a star graph \(\mathcal{G} = (V, E)\) consisting of three edges, where each edge is parameterized by the interval \((0, l_i)\). To validate the accuracy of our numerical method, we examine the non-homogeneous Fokker-Planck equation in the following examples. That is, we replace the Fokker-Planck equation in \eqref{model} with the modified form:
\vspace{-0.3cm}
\begin{equation}
    \frac{\partial \rho_i}{\partial t}-D_i\frac{\partial^2 \rho_i}{\partial x^2}-\frac{\partial }{\partial x}\left(u_i\rho_i\right)=f_i(x,t)\quad i=1,2,3.
\end{equation}
We present two illustrative examples: the first admits polynomial solutions, while the second involves exponential and trigonometric functions. In both cases, we compute the \(l_{\infty}\)-norm errors in approximating the state and control variables, denoted by \(\left(e_{\rho_i}\right)_{K_1}^{K_2}\) and \(\left(e_{u_i}\right)_{K_1}^{K_2}\), respectively, given by

$$
\begin{aligned}
 \left(e_{\rho_i}\right)_{K_1}^{K_2}&=\left\|\rho_i-\left(\rho_i\right)_{K_1}^{K_2}\right\|_{\infty}, \quad
\left( e_{u_i}\right)_{K_1}^{K_2}&=\left\|u_i-\left(u_i\right)_{K_1}^{K_2}\right\|_{\infty},
\end{aligned}
$$ 
where \(\left(\rho_i\right)_{K_1}^{K_2}\) and \(\left(u_i\right)_{K_1}^{K_2}\) represent the approximated optimal values for the state and control variables, respectively, obtained using the collocation method with \(K_1 \times K_2\) collocation points. We also compute the corresponding optimal cost functional values. Let \(\sigma_{K_1}^{K_2}\) denote the approximated optimal cost obtained using \(K_1 \times K_2\) collocation points. Then \(\sigma_{K_1}^{K_2}\) is computed as
\begin{equation*}
\begin{aligned}
    \!\!\!J\approx&\frac{1}{2}\sum_{i=1}^3\Bigg[\int_{0}^{T}\int_{0}^{l_i}\left(\left|(\rho_i)_{K_1}^{K_2}-\rho_i^d\right|^2+\alpha_i\left|(u_i)_{K_1}^{K_2}\right|^2\right) dx dt+\int_{0}^{l_i}\left|(\rho_i(T))_{K_1}^{K_2}-\rho_i^T\right|^2 dx\Bigg]\\
    =&\frac{1}{2}\sum_{i=1}^3\Bigg[\sum_{k_2=1}^{K_2}\int_{\frac{k_2-1}{K_2}}^{\frac{k_2}{K_2}}\sum_{k_1=1}^{K_1}\int_{\frac{k_1-1}{K_1}}^{\frac{k_1}{K_1}}\left(\left|(\rho_i)_{K_1}^{K_2}-\rho_i^d\right|^2+\alpha_i\left|(u_i)_{K_1}^{K_2}\right|^2\right) dx dt\\
    &\hspace{3.5cm}+\sum_{k_1=1}^{K_1}\int_{\frac{k_1-1}{K_1}}^{\frac{k_1}{K1}}\left|(\rho_i(T))_{K_1}^{K_2}-\rho_i^T\right|^2 dx\Bigg],
\end{aligned}    
\end{equation*}
which on using midpoint formula ends up on
\begin{equation}
 \begin{aligned}
\sigma_{K_1}^{K_2}=\frac{1}{2}\sum_{i=1}^3\Bigg[\sum_{k_2=1}^{K_2}\sum_{k_1=1}^{K_1}\left(\left|(\rho_i)_{K_1}^{K_2}\left(x_{k_1},t_{k_2}\right)-\rho_i^d\left(x_{k_1},t_{k_2}\right)\right|^2+\alpha_i\left|(u_i)_{K_1}^{K_2}\left(x_{k_1},t_{k_2}\right)\right|^2\right) \\+\sum_{k_1=1}^{K_1}\left|\left(\rho_i\right)_{K_1}^{K_2}\left(x_{k_1},T\right)-\rho_i^T\left(x_{k_1}\right)\right|^2 dx\Bigg].    
\end{aligned}
\end{equation}

\begin{example}\label{example1}
Consider $l_i=1$, $T=1$, $\alpha_i=1$, $D_i=1$ $(i=1,2,3)$, $u^{min}=-1$ and $u^{max}=1$. The desired state is given as 
\begin{align*}
 \rho_d^1=\rho_d^2=x^2(1-x)t,\quad
    \rho_d^3=x^2(1-x)^2t.
\end{align*}
The initial conditions ($\rho_i^0(x)$), $f_i(x,t)$ and $\rho^i_T(x)$ can be found accordingly.
\end{example}
The optimal cost for this FPOCP is $0$. Tables \ref{table 1} and \ref{table 2} report the approximation errors for the state and control variables, respectively, on each edge of the star graph. It is evident that as the dilation parameters \(J_1\) and \(J_2\) increase with $M_1=M_2=4$, the errors tend to zero. This behavior is expected, as the test case involves polynomial solutions. Consequently, the proposed method yields highly accurate results even for relatively small values of the dilation parameters. Furthermore, Table \ref{table 3} displays the values of the cost functional for increasing values of \(J_1\) and \(J_2\), with \(M_1 = M_2 = 4\) held fixed. The results demonstrate the convergence of the method toward the exact optimal cost. Figure  \ref{figure 2} illustrates the coincidence of  the exact optimal state  with the approximate optimal state solutions for \(J_1 = J_2 = 2\) and \(M_1 = M_2 = 4\). Additionally, Figure \ref{figure 3} shows the approximate optimal control, which converges to $0$, as expected, under the same parameter settings.

\begin{minipage}{0.48\textwidth}
\centering
    \begin{tabular}{|c|c|c|c|}
    \hline
        $J_1$ & $J_2$ & $e_{\rho_1}$=$e_{\rho_2}$  &$e_{\rho_3}$ \\       	
\hline
2.0 &	2.0 &	2.2673e-09 &		9.1677e-10 \\	
\hline
2.0 &	3.0 &	1.5190e-10 &		6.1194e-11 \\	
\hline
3.0 &	2.0 &	6.0020e-10 &		2.9207e-10 \\	
\hline
3.0 &	3.0 &	4.5231e-11 &		2.4457e-11 \\	
\hline
    \end{tabular}
    \captionof{table}{$l_\infty$ errors $e_{\rho_i}$ for the state \\variable $\rho_i$ for Example \ref{example1}.}
    \label{table 1}  
\end{minipage}
 \hspace{-0.4cm} 
\begin{minipage}{0.48\textwidth}
\centering
    \begin{tabular}{|c|c|c|c|}
    \hline
        $J_1$ & $J_2$ & $e_{u_1}$=$e_{u_2}$  &$e_{u_3}$ \\       	

\hline
2.0 &	2.0 &	1.3328e-06 &		1.3209e-06 \\	
\hline
2.0 &	3.0 &	9.8293e-08 &		9.8142e-08 \\	
\hline
3.0 &	2.0 &	1.5256e-07 &		1.5099e-07 \\	
\hline
3.0 &	3.0 &	1.1494e-08 &		1.1389e-08\\ 	
\hline
    \end{tabular}
    \captionof{table}{$l_\infty$ errors $e_{\rho_i}$ for the control \\variable $u_i$ for Example \ref{example1}.}
    \label{table 2}
\end{minipage}


\begin{figure}
\vspace{-2.5cm} 
\centering
\begin{minipage}{.5\textwidth}
  \centering
  \includegraphics[width=0.99\linewidth]{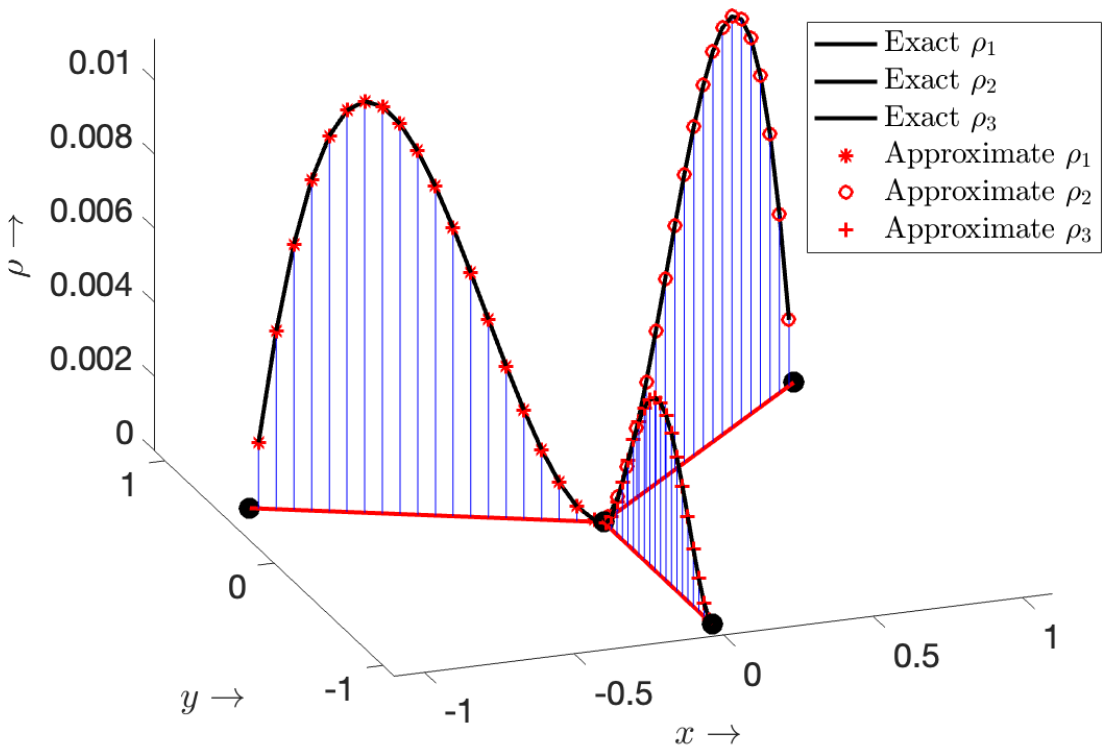}
\caption*{}
  \label{fig3}
\end{minipage}%
\begin{minipage}{.5\textwidth}
  \centering
  \includegraphics[width=0.99\linewidth]{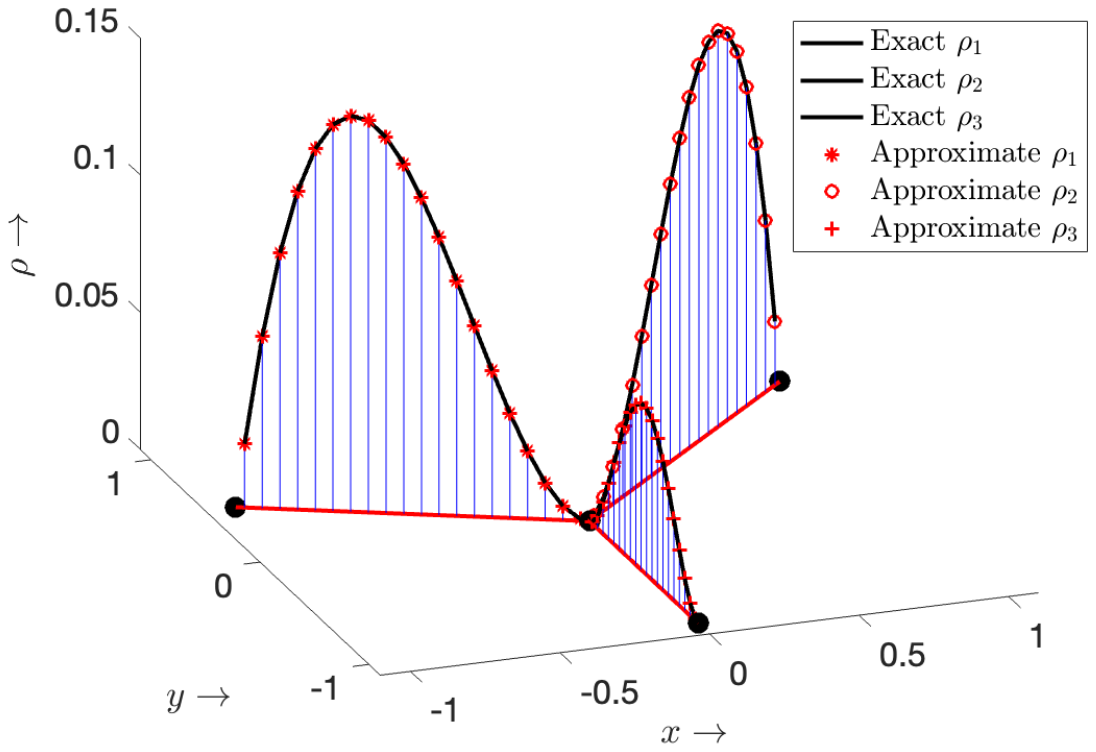}
 \caption*{}
  \label{fig4}
\end{minipage}

\setlength{\abovecaptionskip}{-13pt} 
    \setlength{\belowcaptionskip}{-13pt} 
    \vspace{-3cm}
\caption{Exact and approximate optimal state on the metric graph at time $t_1=0.075$ (on the left) and $t_2=0.975$ (on the right) for $J_1=J_2=2$ and $M_1=M_2=4$ for Example \ref{example1}. }
\label{figure 2}
\end{figure}
 
\begin{figure}
\centering
\begin{minipage}{.5\textwidth}
  \centering
  \includegraphics[width=0.99\linewidth]{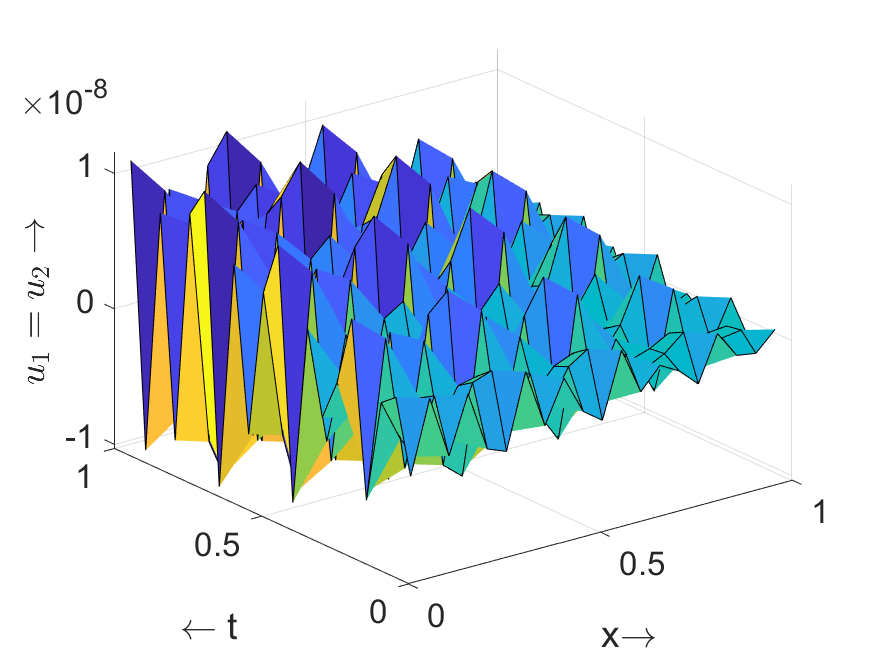}
\caption*{}
  \label{fig5}
\end{minipage}%
\begin{minipage}{.5\textwidth}
  \centering
  \includegraphics[width=0.99\linewidth]{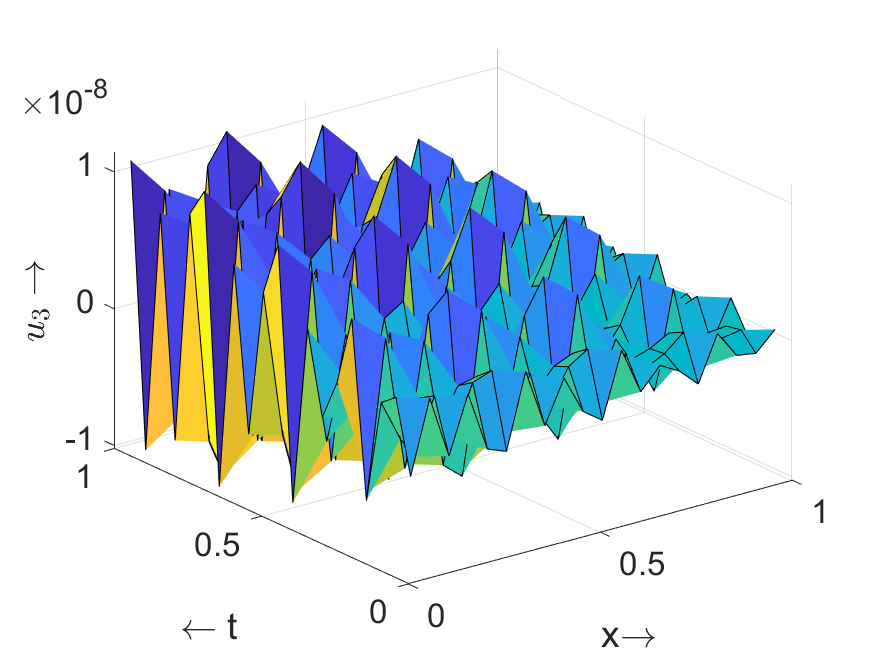}
 \caption*{}
  \label{fig6}
\end{minipage}
\setlength{\abovecaptionskip}{-13pt} 
    \setlength{\belowcaptionskip}{-13pt} 
\caption{Approximate optimal control $(u_1=u_2)$ (on the left) and $u_3$ (on the right) for $J_1=J_2=3$ and $M_1=M_2=4$ for Example \ref{example1}. }
\label{figure 3}
\end{figure}
Next we consider an example involving exponential and trigonometric functions to better verify the accuracy of the scheme.
\begin{example}\label{example2}
    
Consider $l_i=1$, $T=1$, $\alpha_i=1$, $D_i=\frac{1}{4\pi^2}$ $(i=1,2,3)$, $u^{min}=-1$ and $u^{max}=1$. The desired state is given as 
\begin{align*}
 \rho_d^1=\rho_d^2=e^{-t}\sin^2(\pi x), \quad
    \rho_d^3=2e^{-t}\sin^2(\pi x).
\end{align*}
The initial conditions ($\rho_i^0(x)$), $f_i(x,t)$ and $\rho^i_T(x)$ can be found accordingly.
\end{example}
\begin{minipage}{0.3\textwidth}
\begin{tabular}{|c|c|c|}
    \hline
        $J_1$ & $J_2$ & $\sigma$ \\       	
\hline
2.0 &	2.0 &	3.9418e-14 \\	
\hline
2.0 &	3.0 &	7.4991e-16 \\	
\hline
3.0 &	2.0 &	1.7851e-15 \\	
\hline
3.0 &	3.0 &	3.5298e-17 \\	
\hline
    \end{tabular}
    \captionof{table}{Approximate value of the optimal cost $(\sigma)$ for Example \ref{example1}.}
    \label{table 3}
    \centering
\begin{tabular}{|c|c|c|}
    \hline
        $J_1$ & $J_2$ & $\sigma$ \\       	
        \hline
2.0 &	2.0 &	2.4823e-03 \\	
\hline
2.0 &    3.0 &    2.4558e-03 \\
\hline
 2.0 &	4.0 &	2.4554e-03 \\
\hline
3.0 &	2.0 &	9.2552e-06 \\
\hline
3.0 &    3.0 &    8.2941e-06 \\
\hline
3.0 &	4.0 &	8.2961e-06 \\
\hline
4.0 &	2.0 &	1.7771e-07 \\	
\hline
4.0 &    3.0 &    7.3355e-08 \\
\hline
4.0 &	4.0 &	6.0509e-08\\
\hline
\end{tabular}
 \captionof{table}{Approximate value of the optimal cost $(\sigma)$ for Example \ref{example2}.}
    \label{table 6}
     \end{minipage}
\hspace{0.1cm}
\begin{minipage}{0.6\textwidth}
 \centering
    \begin{tabular}{|c|c|c|c|}
    \hline
        $J_1$ & $J_2$ & $e_{\rho_1}=e_{\rho_2}$  &$e_{\rho_3}$ \\       	
\hline
2.0 &	2.0 &	1.5119e-02 &		7.0703e-02 \\	
\hline
2.0 &	3.0 &	1.5387e-02 &		6.9981e-02 \\	
\hline
2.0 &	4.0 &	1.5412e-02 &		6.9934e-02 \\	
\hline
3.0 &	2.0 &	8.4354e-04 &		3.8884e-03 \\	
\hline
3.0 &	3.0 &	8.5069e-04 &		3.7389e-03 \\	
\hline
3.0 &	4.0 &	8.4849e-04 &		3.7391e-03 \\	
\hline
4.0 &	2.0 &	5.2537e-05 &		3.7957e-04 \\	
\hline
4.0 &	3.0 &	5.3268e-05 &		3.1820e-04 \\	
\hline
4.0 &	4.0 &	5.6517e-05 &		3.1003e-04 \\	
\hline
    \end{tabular}
    \captionof{table}{$l_\infty$ errors $e_{\rho_i}$ for the state variable $\rho_i$ for Example \ref{example2}.}
    \label{table 4}
    \centering
    \begin{tabular}{|c|c|c|c|}
    \hline
        $J_1$ & $J_2$ & $e_{u_1}$=$e_{u_2}$  &$e_{u_3}$ \\   \hline
2.0 &	2.0 &	2.2081e-02 &		1.0780e-01 \\	
\hline
2.0 &	3.0 &	2.1668e-02 &		1.0824e-01 \\	
\hline
2.0 &	4.0 &	2.1631e-02 &		1.0825e-01 \\	
\hline
3.0 &	2.0 &	9.9663e-04 &		5.8444e-03 \\	
\hline
3.0 &	3.0 &	1.0037e-03 &		5.5009e-03 \\	
\hline
3.0 &	4.0 &	1.0038e-03 &		5.5010e-03 \\	
\hline
4.0 &	2.0 &	3.1768e-04 &		1.7780e-03 \\	
\hline
4.0 &	3.0 &	2.0587e-04 &		1.0530e-03 \\	
\hline
4.0 &	4.0 &	1.6385e-04 &	7.4550e-04 \\	
\hline

    \end{tabular}
    \captionof{table}{$l_\infty$ errors $e_{u_i}$ for the control variable $u_i$ for Example \ref{example2} .}
    \label{table 5}
\end{minipage}

\begin{figure}
\vspace{-2.5cm}
\centering
\begin{minipage}{.5\textwidth}
  \centering
  \includegraphics[width=0.99\linewidth]{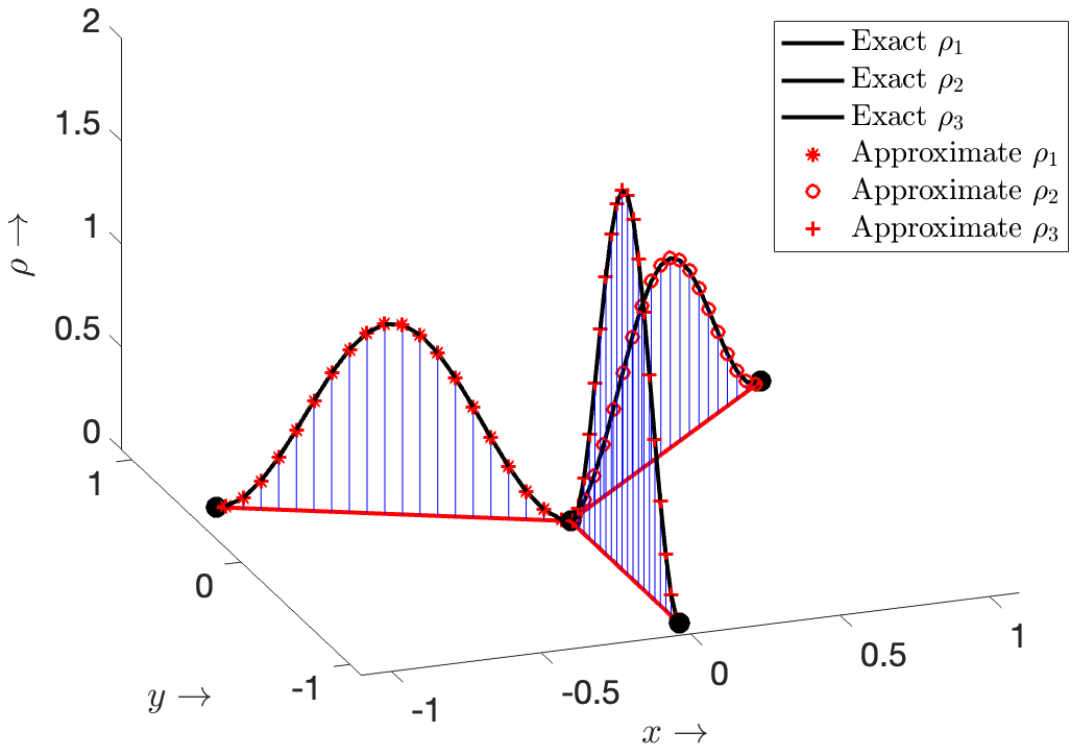}
\caption*{}
  \label{fig7}
\end{minipage}%
\begin{minipage}{.5\textwidth}
  \centering
  \includegraphics[width=0.99\linewidth]{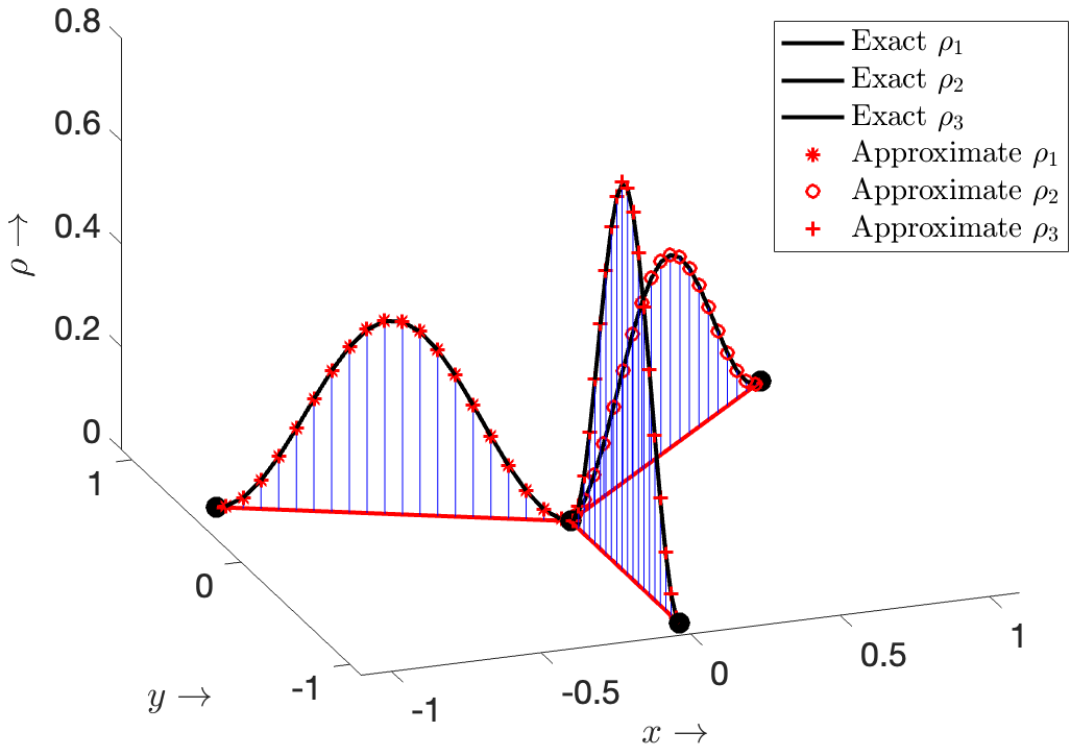}
 \caption*{}
  \label{fig8}
\end{minipage}
\setlength{\abovecaptionskip}{-13pt} 
    \setlength{\belowcaptionskip}{-13pt} 
    \vspace{-3cm}
\caption{Exact and approximate optimal state on the metric graph at time $t_1=0.075$ (on the left) and $t_2=0.975$ (on the right) for $J_1=J_2=3$ and $M_1=M_2=4$ for Example \ref{example2}. }
\label{figure 4}
\end{figure}
The exact optimal cost for this FPOCP is $0$. Tables \ref{table 4} and \ref{table 5} report the approximation errors for the state and control variables, respectively, on each edge of the star graph. As the dilation parameters \(J_1\) and \(J_2\) increase with \(M_1 = M_2 = 4\), the errors clearly decrease and tend toward zero. These results demonstrate that the proposed method achieves satisfactory accuracy, even when approximating complex functions such as trigonometric and exponential functions. Furthermore, Table \ref{table 6} displays the values of the cost functional for increasing values of \(J_1\) and \(J_2\), with \(M_1 = M_2 = 4\) held fixed. The results demonstrate the convergence of the method toward the exact optimal cost. Figure \ref{figure 4}  illustrates the coincidence of  the exact optimal state  with the approximate optimal state solutions for \(J_1 = J_2 = 4\) and \(M_1 = M_2 = 4\). Additionally, Figure \ref{figure 6} shows the approximate optimal control, which converges to $0$, as expected, under the same parameter settings.

\begin{figure}
\centering
\begin{minipage}{.5\textwidth}
  \centering
  \includegraphics[width=0.99\linewidth]{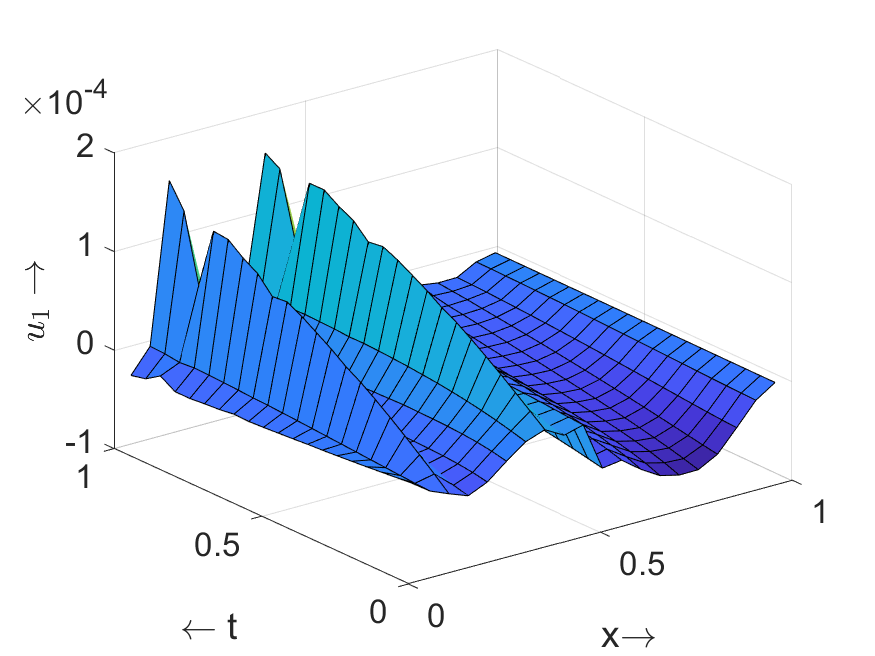}
\caption*{}
  \label{fig11}
\end{minipage}%
\begin{minipage}{.5\textwidth}
  \centering
  \includegraphics[width=0.99\linewidth]{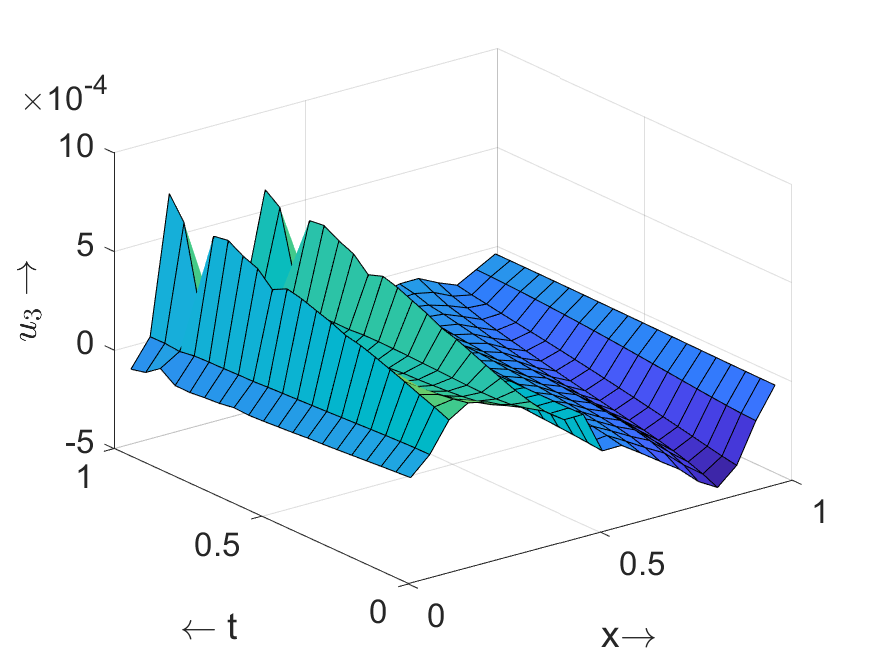}
 \caption*{}
  \label{fig12}
\end{minipage}
\setlength{\abovecaptionskip}{-13pt} 
    \setlength{\belowcaptionskip}{-13pt} 
\caption{Approximate optimal control $(u_1=u_2)$ (on the left) and $u_3$ (on the right) for $J_1=J_2=4$ and $M_1=M_2=4$ for Example \ref{example2}. }
\label{figure 6}
\end{figure}
\section{Conclusion}\label{conc}
This paper addresses an optimal control problem governed by the Fokker-Planck equation (FPOCP) on metric star graphs. The existence and uniqueness of solutions to the nonlinear, nonhomogeneous Fokker-Planck equation were established for any arbitrary admissible control function $u$. Furthermore, the existence of an optimal solution to the FPOCP was proved. The corresponding adjoint equation and the first-order necessary optimality condition were derived for the FPOCP.  Further, to numerically solve the problem, a scheme based on shifted Legendre wavelets was developed, utilizing the exact integrals of the shifted Legendre scaling function basis. A collocation method was employed to solve two illustrative examples, where the $l_\infty$-norm errors were computed for both the optimal state and optimal control approximations. Additionally, the approximate values of the optimal cost functional were presented, demonstrating an approximation error of order lesser than $10^{-8}$. These results confirm the high accuracy and effectiveness of the proposed method for solving the FPOCP.

\section*{Acknowledgements}
 CK acknowledges support from the Natural Sciences and Engineering Research Council of Canada (NSERC), Discovery Grant RGPIN-2025-05864 and the Discovery Launch Supplement DGECR-2025-00184.  
\bibliographystyle{abbrv}
\bibliography{references}
	
\end{document}